\documentclass{article}

\usepackage{inputenc}[utf8]
\usepackage[OT1]{fontenc}
\usepackage{amsmath,amssymb}
\usepackage{amsthm}
\usepackage[normalem]{ulem}
\usepackage[pdftex,bookmarks=true]{hyperref}
\usepackage[noadjust]{cite}
\usepackage[capitalise, noabbrev, nameinlink]{cleveref}
\usepackage{enumitem}
\usepackage{xcolor}
\usepackage{tikz}
\tikzset{
every node/.style={draw, circle, inner sep=2pt}
}

\usepackage{soul}
\usepackage{cancel}

\newtheorem{theorem}{Theorem}[section]
\newtheorem{lemma}[theorem]{Lemma}
\newtheorem{proposition}[theorem]{Proposition}
\newtheorem{corollary}[theorem]{Corollary}

\theoremstyle{definition}
\newtheorem{definition}[theorem]{Definition}

\newtheorem{remark}[theorem]{Remark}
\newtheorem{example}[theorem]{Example}

\definecolor{mydarkgreen}{RGB}{0,100,0} % RGB values

\newcommand{\trans}{^\top}
\newcommand{\adj}{^{\rm adj}}
\newcommand{\cof}{^{\rm cof}}
\newcommand{\inp}[2]{\left\langle#1,#2\right\rangle}

\newcommand{\bzero}{\mathbf{0}}
\newcommand{\bone}{\mathbf{1}}
\newcommand{\ba}{\mathbf{a}}
\newcommand{\bb}{\mathbf{b}}

\newcommand{\be}{\mathbf{e}}

\newcommand{\bu}{\mathbf{u}}

\newcommand{\bw}{\mathbf{w}}
\newcommand{\tr}{\operatorname{tr}}
\newcommand{\nul}{\operatorname{null}}

\newcommand{\range}{\operatorname{range}}
\newcommand{\Col}{\operatorname{Col}}

\newcommand{\spec}{\operatorname{spec}}
\newcommand{\vspan}{\operatorname{span}}

\newcommand{\mat}[1][n]{\operatorname{Mat}_{#1}(\mathbb{R})}
\newcommand{\msym}[1][n]{\operatorname{Sym}_{#1}(\mathbb{R})}
\newcommand{\msymz}[1][n]{\operatorname{Sym}_{\bone}^{#1}(\mathbb{R})}
\newcommand{\mskew}[1][n]{\operatorname{Skew}_{#1}(\mathbb{R})}
\newcommand{\mptn}{\mathcal{S}}

\newcommand{\Paw}{\mathsf{Paw}}

\title{The Strong Spectral Property and the Jacobian Method for Weighted Laplacian Matrices}
\author{
Minerva Catral
\thanks{Department of Mathematics, Xavier University, Cincinnati, OH 45207, USA 
(catralm@xavier.edu).}
\and 
Shaun Fallat
\thanks{Department of Mathematics and Statistics, University of Regina, Regina, SK, S4S 0A2, CA 
(shaun.fallat@uregina.ca).}
\and
Himanshu Gupta
\thanks{Department of Mathematics and Statistics, University of Regina, Regina, SK, S4S 0A2, CA
(himanshu.gupta@uregina.ca).}
\and
Jephian C.-H.~Lin
\thanks{Department of Applied Mathematics, National Yang Ming Chiao Tung University, Hsinchu City 300093, Taiwan 
(jephianlin@gmail.com).}
}

\date{\today}

\begin{document}

\maketitle

\begin{abstract}
Strong matrix properties, roughly speaking, refer to generic conditions on a matrix such that its spectral perturbation and pattern perturbation interact nicely to cover a neighborhood in the ambient space.  With a rich history, these strong properties originate from various fields, including the inverse eigenvalue problem, the sign pattern problem, and structural graph theory.  In this paper, we introduce a new strong property, the strong spectral property for weighted Laplacian matrices (SSPWL), and establish the corresponding Supergraph and Bifurcation lemmas.  Instead of the space of symmetric matrices, the SSPWL considers the ambient space spanned by all weighted Laplacian matrices.  Moreover, we provide a detailed study comparing the Jacobian Method and some strong properties, leading to a full understanding between these two techniques used in different problems.  Using these tools, we identify the potential boundaries of the spectral regions of weighted Laplacian matrices associated with connected graphs on $4$ vertices, extending the analysis from the previous work [S. M. Fallat, H. Gupta, and J. C.-H. Lin. Inverse eigenvalue problem for Laplacian matrices of a graph. \emph{SIAM J.\ Matrix Anal.\ Appl.}, 46:1866–1886, 2025]. In addition, this analysis can be used to identify the absolute algebraic connectivity of such small ordered graphs, and we establish the existence of strong weighted Laplacian matrices for several graph families. 
\end{abstract}  


\medskip

\noindent{\bf AMS subject classifications:}
%05C38, %%Paths and cycles
05C50, %%Graphs and linear algebra (matrices, eigenvalues, etc.)
% 05C57, %%Games on graphs
%05C75, %%Structural characterization of families of graphs
% 05C83, %%Graph minors
%05C85, %%Graph algorithms
% 15A03, %%Vector spaces, linear dependence, rank
15A18, %%Eigenvalues, singular values, and eigenvectors
05C22, %%Signed and weighted graphs
% 15A29, %%Inverse problems
%15B35, %%Sign pattern matrices
15B57, %%Hermitian, skew-Hermitian, and related matrices
%58C15, %%Implicit function theorems; global Newton methods
65F18. %%Inverse eigenvalue problems
%68R10. %%Graph theory (including graph drawing)

\section{Introduction}

The strong properties, though they appeared as different names over time, have a long impact on matrix theory and its applications to the inverse eigenvalue problem, the sign pattern problem, and structural graph theory.  Roughly speaking, the strong properties are generic conditions on a matrix to ensure its spectral perturbation (that preserves the spectrum, ordered multiplicity list, or inertia, etc.) and the pattern perturbation (that preserves signs or zero/nonzero pattern of the entries) interact nicely so that every nearby matrices in the ambient space can be reached by a combination of these two perturbations.  A typical example is the strong spectral property in the inverse eigenvalue problem of a graph (IEP-$G$).  Let $G$ be a graph on $n$ vertices and let $\mptn(G)$ be the set of all $n\times n$ real symmetric matrices whose off-diagonal $i,j$-entry is nonzero if and only if $\{i,j\}\in E(G)$, while the diagonal entries can be any real number.  We consider two perturbations.  The spectral perturbation is $A\mapsto Q\trans AQ$, where $Q$ is any orthogonal matrix near $I$.  The pattern perturbation is the perturbation around $A$ in $\mptn(G)$.  Geometrically, the matrix $A$ has the strong spectral property (SSP) if and only if the tangent spaces of the two perturbations generate the ambient space of symmetric matrices.  

By the inverse function theorem, the SSP implies that every matrix near $A$ can be reached by a combination of the two perturbations.  Specifically, the Supergraph lemma states that a spectral perturbation of $A$ can reach any nearby pattern.  In contrast, the Bifurcation lemma states that, using the pattern perturbation, $A$ can reach any nearby spectrum.  The SSP and related theorems were introduced in \cite{gSAP, IEPG2, bifur}.  

Strong matrix properties have a long history and have appeared in various fields.  In 1971, Arnold \cite{Arnold71} studied matrices whose entries depend on given parameters and considered the cases where a small perturbation of the parameters is enough to reach all possible Jordan canonical forms of nearby matrices.  Arnold called such a matrix family depending on the parameters a versal deformation, which can be viewed as the prototype of several strong properties.  One year later, Arnold \cite{Arnold72} introduced the hypothesis of transversality for operators on a membrane. This idea was then adopted by Colin de Verdi\`ere as the strong Arnold hypothesis for operators defined on a manifold and matrices associated with a graph \cite{CdV88}.  Using the strong Arnold hypothesis, Colin de Verdi\`ere defined a graph parameter $\mu(G)$ \cite{CdV,CdVF}, which was later referred to as the Colin de Verdi\`ere parameter.  

The Colin de Verdi\`ere parameter $\mu(G)$ is defined as the largest multiplicity of the second smallest eigenvalue over the discrete Schr\"odinger operators on a graph that satisfies the strong Arnold hypothesis.  Despite its linear algebra nature, the parameter has a strong connection to the topological properties of the graph.  It is known that $\mu(G) \leq 1$ if and only if $G$ is a path, $\mu(G) \leq 2$ if and only if $G$ is outerplanar, $\mu(G) \leq 3$ if and only if $G$ is planar, and $\mu(G) \leq 4$ if and only if $G$ is linklessly embeddable \cite{CdV,CdVF,KLV}.  Moreover, $\mu(G)$ is shown to be minor monotone, which means $\mu(G) \leq \mu(H)$ if $G$ is a minor of $H$.  Therefore, by the graph minor theorem (see, e.g., \cite{DiestelGT}), the minimal forbidden minors of $\mu(G) \leq k$ are a finite set of graphs for any $k$.

The strong Arnold hypothesis is usually called the strong Arnold property (SAP) in the IEP-$G$ literature.  Motivated by the parameter $\mu(G)$, the other minor monotone parameter $\xi(G)$ was developed based on the SAP \cite{BFH3}.  Later, the strong multiplicity property (SMP) and the strong spectral property (SSP) were introduced and related theorems were established to solve the IEP-$G$ for graphs up to $5$ vertices and many graph families \cite{gSAP,IEPG2,bifur}.  For more details on the IEP-$G$ and several strong properties, see, e.g., \cite{IEPGZF22,CGSV25,CS20,ACFLS24}.

The strong properties appeared in many different fields under different names.  In the nonnegative inverse eigenvalue problem (NIEP), Laffey studied the extreme nonnegative matrices and showed that matrices with some kind of strong property cannot be extreme \cite{Laffey98}.  In the sign pattern problem, the Jacobian Method and the Nilpotent-Centralizer Method were introduced to recognize a spectrally arbitrary sign pattern, while the assumptions of these two methods can also be viewed as some strong properties \cite{DJOvdD00,BMOvdD04,GB12,GB13}.  

% https://www-math.nsysu.edu.tw/~chlin/Publications/Slides/JMM2025.pdf

Let $G$ be graph on $n$ vertices. Define $\mptn_L(G)$ to be the set of all weighted Laplacian matrices associated with $G$. That is, $\mptn_L(G)$ consists of all $n\times n$ real symmetric matrices $A = \begin{bmatrix} a_{i,j} \end{bmatrix}$ such that 
\[
    a_{i,j} \begin{cases}
    < 0 & \text{if } \{i,j\}\in E(G), \\
    = 0 & \text{if } \{i,j\}\notin E(G),\ i\neq j, \\
    - \sum_{k: k\sim i}a_{i,k} & \text{if } i = j.
    \end{cases}
\]

Any matrix $A$ in $\mptn_L(G)$ is singular and positive semidefinite (just like the combinatorial/unweighted Laplacian matrix). Moreover, if the eigenvalues of $A$ are $0 \leq \lambda_2 \leq \lambda_3\leq  \ldots\leq \lambda_n$, then the nullity of $A$ (denoted $\nul(A)$) equals the number of connected components of $G$. In particular, $\lambda_2(A) > 0$ if and only if $G$ is connected.

The combinatorial and weighted Laplacian matrices of a graph are central objects in combinatorics and spectral graph theory, serving as primary tools for understanding structural and connectivity properties of graphs. For a fixed graph $G$, we say that a collection of real numbers $0 \leq \lambda_2 \leq \lambda_3\leq  \ldots\leq \lambda_n$, is \emph{Laplacian realizable} if there exists $A \in \mptn_L(G)$ with spec$(A) = \{0,\lambda_2,\lambda_3,\ldots,\lambda_n\}$. 
The inverse eigenvalue problem for Laplacian matrices of a graph (IEPL) was first formulated in \cite{IEPL}. For a fixed graph $G$, it seeks to characterize all Laplacian realizable sequences of real numbers. We refer the reader to \cite{IEPL} for further results and motivation.

In this paper, we introduce a new strong property, the strong spectral property for weighted Laplacian matrices (SSPWL).  We build the Supergraph lemma and the Bifurcation lemma for the SSPWL.  We show that the Jacobian Method also works for weighted Laplacian matrices.  We prove that the Jacobian matrix associated with a matrix $A$ has full rank if and only if $A$ has all eigenvalues distinct and $A$ has the SSPWL.  This provides a full understanding of the Jacobian Method, the Nilpotent-Centralizer Method, and the strong properties.

The paper is organized as follows. 
In \cref{sec:sspwl}, we introduce the SSPWL, the Supergraph lemma, and the Bifurcation lemma.  In \cref{sec:jac}, we describe how the Jacobian Method can also be utilized for the inverse eigenvalue problem for Laplacian matrices of a graph.  We prove that the Jacobian matrix has full rank if and only if the original matrix has all eigenvalues distinct and has the SSPWL.  With the Jacobian Method, we show that for every connected graph $G$, there is a weighted Laplacian matrix relative to $G$ with the SSPWL.
Using these tools, we consider the IEPL for connected graphs on $4$ vertices in \cref{sec:iepl4}.
In \cref{sec:jacall}, we give a comprehensive study on Jacobian Methods and show the following.
\begin{enumerate}
\item In the sign pattern problem, the Jacobian matrix has full rank if and only if the original matrix has its minimal polynomial degree $n$ and has the non-symmetric strong spectral property (nSSP).  
\item In the inverse eigenvalue problem of a graph, the Jacobian matrix has full rank if and only if the original matrix has distinct eigenvalues and has the strong spectral property (SSP).  
\end{enumerate}
Finally, we summarize our work and list some future directions in \cref{sec:concluding_remarks}.

Before moving on to the next section, we introduce some terminology and notation.  We use $\bone$ and $\bzero$ for the all-ones vector and the zero vector, where the order will be clear from the context. A standard basis vector with a unique 1 in the $i$th position is denoted by $e_i$. Let $\mat$ be the vector space of $n\times n$ real matrices equipped with the inner product $\inp{A}{B} = \tr(B\trans A)$.  Let $\msym$ be the subspace of $\mat$ of all symmetric matrices.  We also define  $\msymz = \{A \in \msym: A\bone = \bzero\}$. For $A \in \mat$, we let $\Col(A)$ denote the column space of the matrix $A$. Finally, we let $N$ be the  $(0,1,-1)$ vertex-edge-oriented incidence matrix of some orientation of the edges of $G$.  Then the combinatorial Laplacian matrix of the graph $G$ can be written as $NN\trans$. 

\section{Strong Spectral Property for Weighted Laplacian Matrices}
\label{sec:sspwl}

In this section, we develop, akin to previous such properties, a strong matrix property specific to weighted Laplacian matrices associated with a graph. We refer to this matrix property as the strong spectral property for weighted Laplacian matrices, or SSPWL for short. This property provides a key tool for investigating the IEPL.

As noted above for  any graph $G$, it is well known that $A \in \mptn_L(G)$ is singular and that $A\bone = \bzero$. This common null vector requires special attention when considering the matrix perturbations that result from the so-called strong matrix properties, in part due to the fact that any such desired perturbation must also have the all-ones vector $\bone$ in its null space. To this end, we consider the projection $H_n$ from $\mathbb{R}^n$ onto $\vspan\{\bone\}^{\perp}$. In this case, we have $H_n = I_n - \frac{1}{n}J_n$, where $J_n$ is the $n\times n$ matrix of all ones.  When $n$ is obvious from the context, we write $H = H_n$ and $J=J_n$. 

We begin by defining the strong spectral property for matrices in $S_L(G)$ and then follow with  several facts and consequences concerning this strong property in the remainder of this section.

\begin{definition}
Let $A \in \mptn_L(G)$.  We say $A$ has the \emph{strong spectral property for weighted Laplacian matrices} (or \emph{SSPWL}) if $X = O$ is the only symmetric matrix such that $A\circ X = I\circ X = O$ and $H(AX - XA)H = O$.  
\end{definition}

Clearly, the SSPWL resembles the standard strong spectral property for symmetric matrices (see, e.g., \cite{IEPG2, gSAP}); however, the requirement that $H(AX - XA)H = O$ adds a level of difficulty when trying to determine when such a matrix $A$ has the SSPWL. To address this extra layer of difficulty, we now state some auxiliary propositions that will be used frequently.

\begin{proposition}
\label{prop:hzero}
Let $X$ be an $n\times n$ matrix.  Then the following are equivalent.
\begin{enumerate}[label={\rm(\arabic*)}]
\item $HXH = O$. \label{item:hzero1}
\item $P_1\trans XP_2 = O$, where $P_1$ and $P_2$ are some matrices with $\Col(P_1) = \Col(P_2) = \vspan\{\bone\}^\perp$. \label{item:hzero2}
\item $N_1\trans XN_2 = O$, where $N_1$ and $N_2$ are the vertex-edge incidence matrix of some trees. \label{item:hzero3}
\item $(\be_{i_1} - \be_{j_1})\trans X (\be_{i_2} - \be_{j_2}) = 0$ for all pairs $i_1 \neq j_1$ and $i_2 \neq j_2$. \label{item:hzero4}
\end{enumerate}
\end{proposition}
\begin{proof}
Observe that $\Col(H) = \Col(P_1) = \Col(P_2) = \vspan\{\bone\}^\perp$.  Therefore, there are matrices $L_1$ and $L_2$ such that $HL_1 = P_1$ and $P_1L_2 = H$.  Similarly, there are matrices $R_1$ and $R_2$ such that $HR_1 = P_2$ and $P_2R_2 = H$.  Consequently, if $HXH = O$, then $P_1\trans XP_2 = L_1\trans HXHR_1 = O$, and if $P_1\trans XP_2 = O$, then $HXH = L_2\trans P_1\trans XP_2R_2 = O$.  Thus, \ref{item:hzero1} and \ref{item:hzero2} are equivalent.  

Statement~\ref{item:hzero3} is a special case of Statement~\ref{item:hzero2}. 
Statement~\ref{item:hzero3} is equivalent to $E\trans XE = O$, where $E$ is a matrix with $\binom{n}{2}$ columns of the form $\be_i - \be_j$.  Since $\Col(H) = \Col(E)$, Statement~\ref{item:hzero4} may be viewed as a special case of Statement~\ref{item:hzero2} since the vectors $\be_{i} - \be_{j}$ with $i\neq j$ span the subspace $\{\bone\}^\perp$.
\end{proof}

Now we introduce the verification matrix to check whether a given matrix has the SSPWL.  

\begin{definition}
\label{def:verification}
Let $A\in\mptn_L(G)$ and $P$ a matrix with $\Col(P) = \vspan\{\bone\}^\perp$.  Define the \emph{SSPWL verification matrix of $A$ with respect to $P$} as $\Psi_A$ whose rows are indexed by $E(\overline{G})$ such that the row of $\Psi_A$ corresponding to $e\in E(\overline{G})$ is a vector corresponding to the strictly upper triangular part of $P\trans(AE_e - E_eA)P$.  Here $E_e$ is the matrix whose two entries corresponding to the edge $e$ are $1$ while all other entries are $0$.  By default, the rows and columns of $\Psi_A$ are sorted by lexicographic order.   
\end{definition}

\begin{example}
\label{ex:p4sspwl}
Consider the matrices 
\[
    A = \begin{bmatrix}
        0.5 & -0.5 & 0 & 0 \\
        -0.5 & 1.5 & -1 & 0 \\
        0 & -1 & 2.5 & -1.5 \\
        0 & 0 & -1.5 & 1.5
    \end{bmatrix}
    \text{ and }
    P = \begin{bmatrix}
        1 & 1 & 1 \\
        -1 & 0 & 0 \\
        0 & -1 & 0 \\
        0 & 0 & -1
    \end{bmatrix}
\]
with $A\in\mptn_L(P_4)$.  Let $e_1 = \{1,3\}$, $e_2 = \{1,4\}$, $e_3 = \{2,4\}$, and $E_{e_i}$ be the matrix $4\times 4$ whose entries correspond to $e_i$ are $1$ while all other entries are $0$.  Then any real symmetric matrix $X$ satisfying $A\circ X = I\circ X = O$ can be written as $X = c_1E_{e_1} + c_2E_{e_2} + c_3E_{e_3}$.  By \cref{prop:hzero}, $H(AX - XA)H = O$ if and only if 
\[
    P\trans(AX - XA)P = \sum_{i=1}^3 c_i P\trans(AE_{e_i} - E_{e_i}A)P = O.
\]
Since $P\trans(AE_{e_i} - E_{e_i}A)P$ is a skew-symmetric matrix, we focus on the entries in the strict upper triangular part, which has three entries.  This leads to the verification matrix
\[
    \Psi = \begin{bmatrix}
        2.5 & -0.5 & -3.5 \\
        -1.5 & 0.5 & 2.5 \\
        1.5 & 0.5 & -0.5
    \end{bmatrix}, 
\]
whose rows correspond to $e_1,e_2,e_3$ and columns correspond to the strict upper triangular entries of the matrix $P\trans(AE_{e_i} - E_{e_i}A)P$.  The verification matrix $\Psi$ has a nontrivial vector $(3,4,-1)$ in its left kernel.  Thus, it follows that $X = 3E_{e_1} + 4E_{e_2} - E_{e_3}$ witnesses the failure of the SSPWL of the matrix $A$ above.
\end{example} 

\begin{corollary}
Let $A\in\mptn_L(G)$.  Then $A$ has the SSPWL if and only if $\Psi_A$ has full row rank.  The statement is independent of the choice of $P$.
\end{corollary}

\begin{corollary}
\label{cor:hbij}
The linear map $X\mapsto HXH$ from $\{X\in\msym: I\circ X = O\}$ to $\msymz$ is a bijection. 
\end{corollary}
\begin{proof}
It is straightforward to see that the map is linear.  We first verify that this map is injective by establishing that its kernel is $\{O\}$.  If $HXH = O$, then  
\[
    (\be_i - \be_j)\trans X (\be_i - \be_j) = 0 
\]
by \cref{prop:hzero}, which implies that the $i,j$-entry of $X$ is zero since $I\circ X = O$.  Since $i$ and $j$ are arbitrary, $X = O$.

Since 
\[
    \dim(\{X\in\msym: I\circ X = O\}) = \dim(\msymz) = \binom{n}{2},
\]
this injective map is indeed a bijection.
\end{proof}

\subsection{Manifolds and Perturbations}

We first note that the intuitive extension of the strong spectral property (SSP) does not apply in the case of weighted Laplacian matrices.  Let $A\in\mptn(G)$.  Recall that the strong spectral property \cite{IEPG2, gSAP} is equivalent to the transversal intersection at $A$ between the manifolds
\[
    \begin{aligned}
        \mathcal{M}_1 &= \mptn(G), \text{ and } \\
        \mathcal{M}_2 &= \{Q\trans AQ: Q \text{ orthogonal}\}.
    \end{aligned}
\]
Thus, a natural analog would be the transversality at $A\in\mptn_L(G)$ between $\mptn_L(G)$ and $\mathcal{M}_2$.  Unfortunately, the all-ones matrix $J$ is always in the normal spaces of both $\mptn_L(G)$ and $\mathcal{M}_2$ at $A$.  The tangent space of $\mptn_L(G)$ at $A$ is the span of $\mptn_L(G)$.  Each of the matrix in the tangent space has each row sum equal to zero, so $J$ is orthogonal to the tangent space.  On the other hand, $AJ - JA = O$ indicates that $J$ is in the normal space of $\mathcal{M}_2$ at $A$.  Therefore, in this extended version no matrix will have the strong property.  

The reason for this behavior is because there is too much freedom in the ambient space $\msym$ so that $J$ is always in the normal space.  If we shrink the ambient space, then the normal spaces shrink as well and could possibly intersect trivially.  

Instead, we consider two sets
\[
    \begin{aligned}
        \mathcal{M}'_1 &= \mptn_L(G), \text{ and } \\
        \mathcal{M}'_2 &= \{Q\trans AQ: Q \text{ orthogonal},\ Q\bone = \bone\}.
    \end{aligned}
\]
Observe that there is a slight change from $\mathcal{M}_2$ to $\mathcal{M}'_2$ since $\mathcal{M}_2$ is not included in the new ambient space $\msymz$.  

We now consider some foundational analysis concerning the sets 
 $\mathcal{M}'_1$ 
        $\mathcal{M}'_2$, from the matrix perturbation point of view.  Let $A\in\mptn_L(G)$.  Consider the function  
\begin{equation}
\label{eq:perturb}
    F(B, K) = e^{-K}Ae^{K} + B    
\end{equation}
with codomain $\msymz$, where $B \in \vspan\mptn_L(G)$ and $K$ is a skew-symmetric matrix with $K\bone = \bzero$.  

\begin{proposition}
\label{prop:range} Let $G$ be a graph on $n$ vertices and $A\in\mptn_L(G)$.
Assume $B \in \vspan\mptn_L(G)$ and $K$ is a skew-symmetric matrix satisfying $K\bone = \bzero$.
Let $F$ be the function as defined in \cref{eq:perturb}, having codomain $\msymz$.  Define
\[
    \dot{F} = \dot{F}\big|_{\substack{B = O\\K = O}},\quad 
    F_B = \frac{\partial F}{\partial B} \big|_{\substack{B = O\\K = O}}, \quad\text{and}\quad 
    F_K = \frac{\partial F}{\partial K} \big|_{\substack{B = O\\K = O}}.
\]
Then the following statements hold.
\begin{enumerate}[label={\rm(\arabic*)}]
\item $\range(F_B) = \vspan\mptn_L(G)$. \label{item:range-fb}
\item $\range(F_B)^{\perp} = \{HXH: X\in\msym,\ A\circ X = I \circ X = O\}$. \label{item:range-fbp}
\item $\range(F_K) = \{K\trans A + AK: K\in\mskew,\ K\bone = \bzero\}$. \label{item:range-fk}
\item $\begin{aligned}[t]
    \range(F_K)^{\perp} &= \{Y \in \msymz: AY - YA = O\} \\
    &= \{HXH: X\in\msym,\ H(AX - XA)H = O\}.
\end{aligned}$ \label{item:range-fkp}
\end{enumerate}
\end{proposition}
\begin{proof}
The ranges of $F_B$ and $F_K$ are straightforward.  The computation for the function $e^K$ can be found in \cite[Example~2.1]{bifur}.

Next, we claim that $\range(F_B)^{\perp} = \{HXH: A\circ X = I \circ X = O\}$.  We emphasize that the orthogonal complement is taken with respect to the ambient space $\msymz$.  Let $L\in\mptn_L(G)$ and $X$ a real symmetric matrix such that $A\circ X = I\circ X = O$.  Note that $LH = HL = L$ since the rows and columns of $L$ are in $\vspan\{\bone\}^\perp$.  Then 
\[
    \inp{L}{HXH} = \tr(HXHL) = \tr(HXL) = \tr(XLH) = \tr(XL) = 0.
\]
Since the argument works for any $L\in\mptn_L(G)$, the matrix $HXH$ is in the orthogonal complement of $\range(F_B)$.  On the other hand, we count the dimension to ensure that $\{HXH: A\circ X = I\circ X = O\}$ is the orthogonal complement.  By \cref{cor:hbij}, $X\mapsto HXH$ is an injective map from $\{HXH: A\circ X = I\circ X = O\}$ to $\msymz$.  Thus, $\dim(\{HXH: A\circ X = I\circ X = O\}) = |E(\overline{G})|$.  Since $\dim(F_B) = |E(G)|$ and $\dim(\msymz) = \binom{n}{2} = |E(G)| + |E(\overline{G})|$, the claim follows.

Next, we find $\range(F_K)^{\perp}$.  Let $Y$ be a matrix in $\msymz$ that is orthogonal to $K\trans A + AK$ for any skew-symmetric matrix $K$ with $K\bone = \bzero$.  Then 
\[
    \begin{aligned}
        0 &= \tr(Y(K\trans A + AK)) = \tr(K\trans AY - K\trans YA) \\
        &= \inp{AY - YA}{K}
    \end{aligned}
\]
for any skew-symmetric matrix $K$ with $K\bone = \bone$.  Note that $AY - YA$ is also a skew-symmetric matrix with $(AY - YA)\bone = \bzero$.  By selecting $K = AY - YA$, we conclude that $AY - YA = O$.  By \cref{cor:hbij}, we may write $Y = HXH$ for some $X\in\msym$.  Thus, the orthogonal complement is given by 
\[
    \begin{aligned}
        &\mathrel{\phantom{=}}\{Y \in \msymz: AY - YA = O\} \\
      & = \{HXH: X\in\msym,\ AHXH - HXHA = O\} \\
        & =\{HXH: X\in\msym,\ H(AX - XA)H = O\}.
    \end{aligned}
\]
This completes the proof.
\end{proof}

The next result follows from the previous proposition.

\begin{corollary}
\label{cor:dim} 
Let $F$ be the function defined in \cref{eq:perturb} and follow the same hypotheses as in \cref{prop:range}.
Assume that $\spec(A) = \{\lambda_1^{(m_1)}, \ldots, \lambda_q^{(m_q)}\}$ with $0 = \lambda_1 < \cdots < \lambda_q$.  Then the following statements hold:
\begin{enumerate}[label={\rm(\arabic*)}]
\item $\dim(\range(F_B)) = |E(G)|$. \label{item:dim-fb}
\item $\dim(\range(F_B)^{\perp}) = |E(\overline{G})|$. \label{item:dim-fbp}
\item $\dim(\range(F_K)) = \binom{n}{2} - \binom{m_1}{2} - \sum_{i=2}^q\binom{m_i + 1}{2}$. \label{item:dim-fk}
\item $\dim(\range(F_K)^{\perp}) = \binom{m_1}{2} + \sum_{i=2}^q\binom{m_i + 1}{2}$. \label{item:dim-fkp}
\end{enumerate}
\end{corollary}
\begin{proof}
Statements \ref{item:dim-fb} and \ref{item:dim-fbp} are straightforward.  We focus on \ref{item:dim-fkp}, and note that \ref{item:dim-fk} is an immediate consequence by applying basic properties of orthogonal complementary subspaces.

Diagonalize $A$ as $Q\trans AQ = D$ so that $\frac{1}{\sqrt{n}}\bone$ is the first column of $Q$ and the $1,1$-entry of $D$ is zero.  Now we may write $AY - YA = O$ as $D(Q\trans YQ) - (Q\trans YQ)D = O$.  Moreover, $Y\in\msymz$ is equivalent to $Q\trans YQ$ has its first row and column zero.  This matrix commutes with $D$ if and only if the $i,j$-entry of $Q\trans YQ$ is zero whenever the $i,i$-entry and the $j,j$-entry of $D$ are distinct.  Thus, $Q\trans YQ$ has $\binom{m_1}{2} + \sum_{i=2}^q\binom{m_i + 1}{2}$ free entries.
\end{proof}

\begin{proposition}
\label{prop:rangestrong} 
Let $F$ be the function defined in \cref{eq:perturb} and follow the same hypotheses as in \cref{prop:range}.
Then the following are equivalent. 
\begin{enumerate}[label={\rm(\arabic*)}]
\item $\dot{F}$ is surjective. \label{item:rangestrong-sur}
\item $\range(F) = \range(F_B) + \range(F_K)$. \label{item:rangestrong-sum}  
\item $\range(F_B)^{\perp} \cap \range(F_K)^{\perp} = \{O\}$. \label{item:rangestrong-int}
\item $A$ has the SSPWL. \label{item:rangestrong-sspwl}
\end{enumerate}
\end{proposition}
\begin{proof}
Statements \ref{item:rangestrong-sur}, \ref{item:rangestrong-sum}, and \ref{item:rangestrong-int} are equivalent by applying basic linear algebra facts.

To see that Statements \ref{item:rangestrong-int} and \ref{item:rangestrong-sspwl} are equivalent, we use contrapositive statements.  Suppose $X\in\msym$ is a nonzero matrix such that $A\circ X = I\circ X = O$ and $H(AX-XA)H = O$.  By \cref{cor:hbij}, $HXH\in\msymz$ is a nonzero matrix.  Thus, $HXH$ is a nonzero matrix in $\range(F_B)^{\perp} \cap \range(F_K)^{\perp}$ from the above definitions.

On the other hand, suppose $Y\in\msymz$ is a nonzero matrix in $\range(F_B)^{\perp} \cap \range(F_K)^{\perp}$.  
Since $Y\in\range(F_B)^{\perp}$, there is a matrix $X\in\msym$ such that $A\circ X = I\circ X = O$ and $Y = HXH$.  By \cref{cor:hbij}, $X$ is nonzero.  Since $Y\in\range(F_K)^{\perp}$ and $AH = HA = A$, 
\[
    O = AY - YA = A(HXH) - (HXH)A = H(AX - XA)H.
\]
Therefore, $X$ is a nonzero matrix such that $A\circ X = I\circ X = O$ and $H(AX - XA)H = O$.
\end{proof}

We are now in a position to establish the so-called Supergraph lemma and the Bifurcation lemma associated with this strong matrix property for weighted Laplacian matrices.  Both of these results utilize the inverse function theorem, which is quoted below from \cite[Theorem~2.4]{bifur}.

\begin{theorem}[Inverse function theorem]
\label{thm:inv}
Let $U$ and $W$ be finite-dimensional vector spaces over $\mathbb{R}$.  Let $F$ be a smooth function from an open subset of $U$ to $W$ with $F(\bu_0) = \bw_0$.  If $\dot{F}\Big|_{\bu = \bu_0}$ is surjective, then there is an open subset $W' \subset W$ containing $\bw_0$ and a smooth function $T: W' \rightarrow U$ such that $T(\bw_0) = \bu_0$ and $F\circ T$ is the identity map on the set $W'$.
\end{theorem}

We begin with the Supergraph lemma for weighted Laplacian matrices.

\begin{lemma}[Supergraph lemma]
\label{lem:supergraph}
Let $G$ be a spanning subgraph of the graph $H$.  Suppose $A\in\mptn_L(G)$ has the SSPWL.  Then there is a matrix $A'\in\mptn_L(H)$ such that
\begin{enumerate}[label={\rm(\arabic*)}]
\item $\spec(A') = \spec(A)$, 
\item $A'$ has the SSPWL, and 
\item $\|A' - A\|$ can be chosen arbitrarily small.
\end{enumerate}
\end{lemma}
\begin{proof}
Let $A\in\mptn_L(G)$.  Consider the function $F$ as in \cref{eq:perturb}.  Since $A$ has the SSPWL, $\dot{F}$ is surjective.  By \cref{thm:inv}, for any $M\in\msymz$ close to $A$, there are $B'$ and $K'$ such that $F(B',K') = M$.  Choose $M$ as the matrix obtained from $A$ by perturbing its zero entries in $E(H)\setminus E(G)$ into any small negative values and adjusting the diagonal entries to ensure $M\in\mptn_L(H)$. Thus, $e^{-K'}Ae^{K'} + B' = M$.  When the perturbation from $A$ to $M$ is small enough, the corresponding $B'$ is also small by continuity.  Then $A' = e^{-K'}Ae^K = M - B'$ will have the same pattern as $M$, so $A'\in\mptn_L(H)$. Furthermore, $A'$ has the same spectrum as $A$ by similarity.  Since this perturbation is small, $A'$ still has the SSPWL.
\end{proof}

We now verify the Bifurcation lemma for weighed Laplacian matrices with this strong matrix property.

\begin{lemma}[Bifurcation lemma]
\label{lem:bifurcation}
Let $G$ be a graph.  Suppose $A\in\mptn_L(G)$ has the SSPWL.  Then there is $\epsilon > 0$ such that for any matrix $M\in\msymz$ with $\|M - A\| < \epsilon$, a matrix $A'\in\mptn_L(G)$ exists such that
\begin{enumerate}[label={\rm(\arabic*)}]
\item $\spec(A') = \spec(M)$ and 
\item $A'$ has the SSPWL.
\end{enumerate}
\end{lemma}
\begin{proof}
Consider the function $\tilde{F}(B,K) = e^{-K}(A + B)e^K$.  Observe that $\tilde{F}$ and $F$ are the same when $B = O$ or $K = O$.  Then $\tilde{F}$ also satisfies \cref{prop:range,cor:dim,prop:rangestrong}.  Since $A$ has the SSPWL, the derivative of $\tilde{F}$ is surjective.  By \cref{thm:inv}, for any $M\in\msymz$ close to $A$, there are $B'$ and $K'$ such that $\tilde{F}(B',K') = M$.  Thus, $e^{-K'}(A + B')e^{K'} = M$.  Then, since $B'$ is a small perturbation of $A$, it follows that $A' = A + B' = e^{K'}Me^{-K'}$ is in $\mptn_L(G)$ and has the same spectrum as $M$.  
\end{proof}

\begin{corollary}
\label{cor:tree-SSPWL}
Let $G$ be a tree.  Then $A\in\mptn_L(G)$ has the SSPWL implies that $\spec(A)$ is composed of distinct elements.
\end{corollary}
\begin{proof}
Let $n = |V(G)|$ and $m = |E(G)|$.  Let $A\in\mptn_L(G)$ be a matrix with the SSPWL and $F$ the corresponding function as in \cref{eq:perturb}.  According to \cref{prop:rangestrong}, a necessary condition for $A$ to have the SSPWL is that 
\[
    \dim(\range(F_B)) + \dim(\range(F_K)) \geq \dim(\msymz) = \binom{n}{2}.
\]
By \cref{cor:dim}, we have 
\[
    \dim(\range(F_B)) = m \text{ and } \dim(\range(F_K)) = \binom{n}{2} - \binom{m_1}{2} + \sum_{i=2}^q\binom{m_i + 1}{2}.
\]
Since $G$ is a tree, $m = n - 1$.  Since $G$ is connected, $m_1 = 1$.  Therefore, the necessary condition becomes
\[
    n - 1 + \binom{n}{2} - \sum_{i=2}^q\binom{m_i + 1}{2} \geq \binom{n}{2},
\]
which is equivalent to $n - 1 \geq \sum_{i=2}^q\binom{m_i + 1}{2}$.  This happens only when $m_i = 1$ for all $i$.
\end{proof}

\begin{example}
Let $G=K_4-e$.  The matrix  
\[
   A= \begin{bmatrix}
        4 & 0 & -3 & -1 \\
        0 & 4 & -3 & -1 \\
        -3 & -3 & 7 & -1 \\
        -1 & -1 & -1 & 3
    \end{bmatrix}
\]
is in $\mptn_L(K_4-e)$ and has the spectrum $\{0,4,4,10\}$. By direct computation, or referring to \cref{sec:iepl4}, it follows that $A$ has the SSPWL, and hence, by the Supergraph lemma (\cref{lem:supergraph}) there exists a matrix $A'$ in $\mptn_L(K_4)$ with spectrum $\{0,4,4,10\}$. 

On the other hand, the combinatorial Laplacian matrix $L=nI-J$ is in $\mptn_L(K_n)$, has spectrum equal to $\{0, n^{(n-1)}\}$, and has the SSPWL. Hence, it follows from an application of the Bifurcation lemma (\cref{lem:bifurcation}) that there exists an $A' \in \mptn_L(K_n)$ with the spectrum of $A'$ being $\{0, \lambda_{1}^{(m_1)}, \lambda_{2}^{(m_2)}, \ldots, \lambda_{q}^{(m_q)}\}$ for any $\lambda_i$ close to the original eigenvalue $n$, where $1 \leq q \leq n-1$ and with $\sum_i m_i = n-1$.
\end{example}

\section{Equivalence to the Jacobian Method}
\label{sec:jac}

In this section, we introduce the Jacobian Method for verifying the SSPWL and show that for every connected graph $G$, there is a matrix in $\mptn_L(G)$ with the SSPWL.  The idea of the Jacobian Method for inverse eigenvalue problems was adopted from the sign pattern problem \cite{DJOvdD00,GB13,bifur}, and a key idea, to be verified, is to certify that the nonzero entries of a matrix have full control of the coefficients of the characteristic polynomial and consequently the spectrum.  

We start with an example to demonstrate the Jacobian Method before proceeding with the theoretical details.

\begin{example}
\label{ex:k14jm}
Let $G = K_{1,4}$ and 
\[
    A = \begin{bmatrix}
    \sum_{i=1}^4 w_i & -w_1 & -w_2 & -w_3 & -w_4 \\
    -w_1 & w_1 & 0 & 0 & 0 \\
    -w_2 & 0 & w_2 & 0 & 0 \\
    -w_3 & 0 & 0 & w_3 & 0 \\
    -w_4 & 0 & 0 & 0 & w_4 \\
    \end{bmatrix}.
\]
Let $s_k = s_k(A)$ be the sum of all $k\times k$ minors of $A$.  By direct computation, we have 
\[
    \begin{aligned}
        s_1 &= 2(w_1 + w_2 + w_3 + w_4), \\
        s_2 &= 3(w_1w_2 + w_1w_3 + w_1w_4 + w_2w_3 + w_2w_4 + w_3w_4), \\
        s_3 &= 4(w_1w_2w_3 + w_1w_2w_4 + w_1w_3w_4 + w_2w_3w_4), \\
        s_4 &= 5w_1w_2w_3w_4, \\
    \end{aligned}
\]
and 
\[
    \det(A - xI) = s_0(-x)^5 + s_1(-x)^4 + \cdots + s_5.
\]
Note that $s_1 = 1$ and $s_5 = 0$. Then, the Jacobian matrix $J_A = \begin{bmatrix} \frac{\partial s_i}{\partial w_j} \end{bmatrix}$ is 
\[
    \begin{bmatrix}
        2 & 2 & 2 & 2 \\
        3(w_2 + w_3 + w_4) & \cdots & \cdots & 3(w_1 + w_2 + w_3) \\
        4(w_2w_3 + w_2w_4 + w_3w_4) & \cdots & \cdots & 4(w_1w_2 + w_1w_3 + w_2w_3) \\
        5w_2w_3w_4 & 5w_1w_3w_4 & 5w_1w_2w_4 & 5w_1w_2w_3
    \end{bmatrix}.
\]
If $J_A$ is invertible with some positive $w_1, \ldots, w_4$, then the inverse function theorem implies that $w_i$'s have full control of $s_k$'s, which has the same effect as the Supergraph lemma (\cref{lem:supergraph}) and the Bifurcation lemma (\cref{lem:bifurcation}).  
\end{example}

The graph $K_{1,4}$ has $m = 4$ edges and $\overline{m} = 6$ missing edges.  Note that the Jacobian matrix in \cref{ex:k14jm} is of order $m = 4$, while the verification matrix is of order $\overline{m} = 6$.  At first glance, the Jacobian Method and the SSPWL appear rather different, but they lead to  the same conclusion.  In the following, we show that they are equivalent provided that $A$ has distinct eigenvalues.

Given a square matrix $M\in\mat$, we define $s_k(M)$ to be the sum of all its $k\times k$ principal minors.  Then the  characteristic polynomial of $M$ can be written as  
\[
    \det(M - xI) = s_0(M)(-x)^n + s_1(M)(-x)^{n-1} + \cdots + s_n(M).
\]
For each $k$, define the matrix $\nabla s_k(M)$ whose $i,j$-entry is the partial derivative of $s_k(M)$ with respect to the $i,j$-entry.  (Note that here $M$ is not necessarily symmetric and the $i,j$-entry and the $j,i$-entry are treated as different variables.)  If $M(t)$ is a trajectory of matrices where $s_k(M(t))$ is constant, then by the chain rule, we have 
\[
    0 = \frac{d}{dt}s_k(M(t)) = \inp{\nabla s_k(M(0))}{\dot{M}(0)}.
\]
That is, $\nabla s_k(M(0))$ is orthogonal to the tangent of $M(t)$ at $t = 0$. 

\begin{example}
\label{ex:nsymortho}
Let 
\[
    M = \begin{bmatrix}
        1 & \fbox{2} & 3 \\
        4 & 5 & 6 \\
        7 & 8 & 9
    \end{bmatrix}
    \text{ and }
    E(t) = \begin{bmatrix}
        1 & 0 & 0 \\
        t & 1 & 0 \\
        0 & 0 & 1
    \end{bmatrix}.
\]
For $M$, we have 
\[
    s_2 = (1\cdot 5 - \fbox{2}\cdot 4 ) +  (5\cdot 9 - 6\cdot 8)  +  (1\cdot 9 - 3\cdot 7).
\]
By changing the entry with value $\fbox{2}$, $s_2$ will change correspondingly with a ratio of $-4$.  Following a similar idea for all other entries gives
\[
    \nabla s_2(M) = \begin{bmatrix}
        14 & -4 & -7 \\
        -2 & 10 & -8 \\
        -3 & -6 & 6
    \end{bmatrix}.
\]
Consider the matrix $M(t) = E(t)^{-1} ME(t)$ with $M(0) = M$.  Since similarity does not alter the characteristic polynomial, $s_2$ is constant.  By direct computation, we have 
\[
    M(t) = \begin{bmatrix}
        1+2t & 2 & 3 \\
        4+4t-2t^2 & 5-2t & 6-3t \\
        7+8t & 8 & 9
    \end{bmatrix}
    \text{ and }
    \dot{M}(0) = \begin{bmatrix}
        2 & 0 & 0 \\
        4 & -2 & -3 \\
        8 & 0 & 0
    \end{bmatrix}.
\]
As expected, we have $\inp{\nabla s_2(M)}{\dot{M}(0)} = 0$.
\end{example}

The following proposition, which was proved in \cite{GB12}, provides an interesting and systematic way to compute $\nabla s_k$.  We use the notation $\nabla\det(M - xI)$ to denote a matrix whose $i,j$-entry is the partial derivative of $\det(M - xI)$ with respect to the $i,j$-entry. Let $X\cof$ be the cofactor matrix of $X$, the transpose of the adjugate matrix $X\adj$. 

\begin{proposition}
\label{prop:gradientadj}
Let $M\in\mat$ and $\nabla s_k = \nabla s_k(M)$ for $k = 1, \ldots, n$.  Then the total derivative of $\det(M - xI)$ with respect to the entries is 
\[
    \nabla\det(M - xI) = \nabla s_1(-x)^{n-1} + \nabla s_2(-x)^{n-2} + \cdots + \nabla s_n = (M - xI)\cof.
\] 
\end{proposition}

\begin{example}
Let $M$ be as in \cref{ex:nsymortho}.  Direct computation yields
\[
(M - xI)\cof =    \begin{bmatrix}
        (5-x)(9-x)-48 & -(4(9-x)-42) & 32-7(5-x) \\
        -(2(9-x)-24) & (1-x)(9-x)-21 & -(8(1-x)-14) \\
        12-3(5-x) & -(6(1-x)-12) & (1-x)(5-x)-8
    \end{bmatrix},
\]
which can be expanded as 
\[
    (-x)^2\begin{bmatrix}
        1 & 0 & 0 \\
        0 & 1 & 0 \\
        0 & 0 & 1
    \end{bmatrix} + 
    (-x)^1\begin{bmatrix}
        14 & -4 & -7 \\
        -2 & 10 & -8 \\
        -3 & -6 & 6
    \end{bmatrix} + 
    (-x)^0\begin{bmatrix}
        -3 & 6 & -3 \\
        6 & -12 & 6 \\
        -3 & 6 & -3
    \end{bmatrix}.
\]
Thus,  $\nabla s_2(M)$ is  shown as the coefficient of $(-x)^1$.  
\end{example}

\begin{lemma}
\label{lem:gradientind}
Let $M\in\mat$ and $\nabla s_k = \nabla s_k(M)$ for $k = 1, \ldots, n$.  Then $
\{\nabla s_1, \ldots, \nabla s_n\}$ and $\{I, M\trans, \ldots, (M\trans)^{n-1}\}$ span the same subspace.  Consequently, $
\{\nabla s_1, \ldots, \nabla s_n\}$ is linearly independent if and only if the minimal polynomial of $M$ has degree $n$.  
\end{lemma}
\begin{proof}
To simplify notation, let $M_x = M - xI$, treated as a matrix over the field of rational functions $\mathbb{R}(x)$, and consider its characteristic polynomial 
\[
    \det(M_x - tI) = p_0(x)(-t)^n + p_1(x)(-t)^{n-1} + \cdots + p_n(x).
\]
By the Cayley--Hamilton theorem, the inverse of $M_x$ can be expressed as 
\[
    M_x^{-1} = \frac{1}{p_n(x)}(p_0(x)(-M_x)^{n-1} + p_1(x)(-M_x)^{n-2} + \cdots + p_{n-1}(x)I).
\]
Now, it is known that 
\[
    M_x\cof = (M_x\adj)\trans = \det(M_x)(M_x\trans)^{-1}.
\]
Thus, since $p_0(x) = 1$ and $p_n(x) = \det(M_x)$, we have 
\[
    M_x\cof = (-M_x\trans)^{n-1} + p_1(x)(-M_x\trans)^{n-2} + \cdots + p_{n-1}(x)I.
\]
Expanding the powers of $-M_x\trans$ on the right-hand side, we may write 
\[
    M_x\cof = r_0(x)(M\trans)^{n-1} + r_1(x)(M\trans)^{n-2} + \cdots + r_{n-1}(x)I
\]
for some coefficients $r_0(x), \dots, r_{n-1}(x)$.
Equating this expression  for $M_x\cof$ with the formula in \cref{prop:gradientadj}, each element in $\{\nabla s_1, \ldots, \nabla s_n\}$ is in the span of $\{I, M\trans, \ldots, (M\trans)^{n-1}\}$.  

On the other hand, we claim that $r_k(x)$ is a polynomial of degree $k$.  This would imply that $\{r_0(x), \ldots, r_{n-1}(x)\}$ is a basis of the space of polynomials of degree less than $n$, and  therefore, each element of $\{1, (-x), \ldots, (-x)^{n-1}\}$ can be written as a linear combination of $\{r_0(x), \ldots, r_{n-1}(x)\}$.  Consequently, by comparing the two expressions for  $M_x\cof $, each element in $\{I, M\trans, \ldots, (M\trans)^{n-1}\}$ is  in the span of $\{\nabla s_1, \ldots, \nabla s_n\}$, showing that the  two sets  have the same span.   We now examine the leading term of $r_k(x)$.    By direct computation, 
\[
    r_k(x) = \sum_{i=0}^{k} p_i(x) \cdot (-1)^{n-1-i} \binom{n-1-i}{k-i}(-x)^{k-i}.
\]
Note that $p_i(x)$ is the sum of all $i\times i$ principal minors of $M_x$, so its leading term is $\binom{n}{i}(-x)^i$.  Thus, the leading term of $r_k(x)$ is 
\[
    (-x)^{k}\sum_{i=0}^{k} \binom{n}{i} \cdot (-1)^{n-1-i} \binom{n-1-i}{k-i} = (-1)^{n-1}x^{k}\sum_{i=0}^{k} \binom{n}{i}\binom{-(n-k)}{k-i}.
\]
Observe that the summation is the coefficient of $y^k$ in $(1+y)^n(1+y)^{-(n-k)} = (1+y)^k$, which is $1$.   Therefore, the leading term of $r_k(x)$ is $(-1)^{n-1}x^k$.  This completes our argument showing that $
\{\nabla s_1, \ldots, \nabla s_n\}$ and $\{I, M\trans, \ldots, (M\trans)^{n-1}\}$ span the same subspace.

If the minimal polynomial is of degree less than $n$, then the dimension of the span of $\{I, M\trans, \ldots, (M\trans)^{n-1}\}$ is less than $n$, indicating that $\{\nabla s_1, \ldots, \nabla s_n\}$ is linearly dependent.  If  the minimal polynomial is of degree $n$, then the dimension of the span of $\{I, M\trans, \ldots, (M\trans)^{n-1}\}$ is $n$, implying that $\{\nabla s_1, \ldots, \nabla s_n\}$ is a linearly independent set.
\end{proof}

We need one more proposition before we provide the connection between strong properties and the Jacobian matrix.  

\begin{proposition}
\label{prop:tnequiv}
Let $T_1$ and $T_2$ be subspaces of a finite dimensional ambient space $W$.  Let $\alpha$ be a basis of $T_2$ and $\beta$ a basis of $N_2 = T_2^\perp$.  Then the following statements are equivalent.  
\begin{enumerate}[label={\rm(\arabic*)}]
\item $T_1 + T_2 = W$. \label{lab:tnqeuiv-sum}  
\item The orthogonal projection from $N_2$ to $T_1$ is injective.  \label{lab:tnqeuiv-inj}
\item The matrix $\begin{bmatrix} \inp{\ba_i}{\bb_j} \end{bmatrix}$ for $\ba_i\in\alpha$ and $\bb_j\in\beta$ has full column rank. \label{lab:tnqeuiv-mtx}
\end{enumerate}
\end{proposition}
\begin{proof}
The equivalence between \ref{lab:tnqeuiv-sum} and \ref{lab:tnqeuiv-inj} is straightforward.  Next, we verify that \ref{lab:tnqeuiv-inj} and \ref{lab:tnqeuiv-mtx} are equivalent.  Let $A = \begin{bmatrix} \inp{\ba_i}{\ba_j} \end{bmatrix}$ for $\ba_i,\ba_j\in\alpha$ and $P = \begin{bmatrix} \inp{\ba_i}{\bb_j} \end{bmatrix}$ for $\ba_i\in\alpha$ and $\bb_j\in\beta$.  Note that $A$ is invertible since $\alpha$ is a basis.  Thus, $A^{-1}P$ is the matrix representation of the orthogonal projection from $N_2$ to $T_1$ with respect to the bases $\beta$ and $\alpha$.  Therefore, the projection is injective if and only if $A^{-1}P$ has full column rank.  Since $A$ is invertible, $A^{-1}P$ has full column rank if and only if $P$ has full column rank.
\end{proof}

\subsection{Jacobian Method and the SSPWL}

Given a matrix $A\in\mptn_L(G)\subseteq\msymz$, the SSPWL considers two sets:  $\mathcal{M}_1 = \mptn_L(G)$ and  $\mathcal{M}_2$,  the set of all matrices of the form $Q\trans AQ$, where $Q$ is any orthogonal matrix with $Q\bone = \bone$.  The Jacobian Method relies on the Jacobian matrix 
\[
    J_A = \begin{bmatrix}
        \frac{d(s_1,\ldots, s_{n-1})}{d({\rm free\ entries})}
    \end{bmatrix}.
\]
Note that $s_n = 0$ for any matrix in $\mptn_L(G)$ and $\frac{d s_n(A(t))}{dt} = 0$ if $A(t)\in\mptn_L(G)$ for all $t$, so the Jacobian matrix has fewer rows in the weighted Laplacian problem.  Also note that the free entries here indicate the changes $A + tE_{e_j}$ where $e_j$ is an edge of $G$ and $E_{e_j}$ is the matrix whose $2\times 2$ induced principal submatrix on $e_j$ is 
\[
    \begin{bmatrix}
        1 & -1 \\
        -1 & 1
    \end{bmatrix}
\]
and other entries are zero.  By the chain rule, 
\[
    J_A = \begin{bmatrix}
        \inp{\nabla s_i}{E_{e_j}}
    \end{bmatrix}.
\]

For SSPWL, we need one extra fact to reach our goal, utilizing properties of symmetric matrices.  Recall that $H = I - \frac{1}{n}J$.  

\begin{proposition}
\label{prop:gradientprojection}
Let $G$ be a connected graph and $A\in\mptn_L(G)$.  Let $\nabla s_k = \nabla s_k(A)$.  Then $\nabla s_n$ is a multiple of $J$, so $H\nabla s_n H = O$.  Moreover, the collection of matrices $\{H\nabla s_1 H, H\nabla s_2 H, \ldots, H\nabla s_{n-1} H\}$ spans a subspace in $\msymz$, and its dimension equals $n-1$ if and only if the eigenvalues of $A$ are distinct.
\end{proposition}
\begin{proof}
Since $G$ is connected, $\nul(A) = 1$ and the kernel of $A$ is $\vspan\{\bone\}$.  By applying \cref{prop:gradientadj} with $x = 0$, we have $\nabla s_n = A\cof$.  Moreover, $A\cof$ is a multiple of $J$ since $AA\cof = AA\adj = \det(A) I = O$ and $A\cof$ is symmetric.  Necessarily, $H\nabla s_n H = O$. Since $HMH$ is the orthogonal projection of a symmetric matrix $M$ onto $\msymz$, the set $\{H\nabla s_1 H, H\nabla s_2 H, \ldots, H\nabla s_{n-1} H\}$ spans a subspace in $\msymz$.

Let $A = \sum_{i=1}^q\lambda_q P_i$ be the spectral decomposition of $A$, where $\lambda_1, \ldots, \lambda_q$ are the distinct eigenvalues of $A$ while $P_i$ is the projection matrix onto the corresponding eigenspace.  Without loss of generality, we may assume $\lambda_1 = 0$ and $P_1 = \frac{1}{n}J$.  Note that $P_1P_k = P_kP_1 = O$ for any $k \neq 1$ implies that the rows and columns are orthogonal to $\bone$ for any $k \neq 1$.  Therefore, $HP_kH = P_k$ for $k \neq 1$.  

Also, it is known that $\{P_1, P_2, \ldots, P_q\}$, $\{I, A, \ldots, A^{n-1}\}$, $\{\nabla s_1, \nabla s_2, \ldots, \nabla s_n\}$ span the same subspace in $\msym$ by \cref{lem:gradientind}.  By applying $M \mapsto HMH$ to every element, along with the facts that $HP_1H = O$, $HP_kH = P_k$ for $k\neq 1$, and $H\nabla s_n H = O$, we have
\[
    \vspan\{P_2, P_3, \ldots, P_q\} = \vspan\{H\nabla s_1 H, H\nabla s_2 H, \ldots, H\nabla s_{n-1} H\}.
\]
Moreover, its dimension is $n-1$ if and only if $q = n$, that is, the eigenvalues of $A$ are distinct.
\end{proof}

\begin{theorem}
\label{thm.SSPWL-Jac}
Let $G$ be a connected graph,  $A\in\mptn_L(G)$ and $J_A$ the associated Jacobian matrix.  Then   $J_A$ has full rank if and only if the eigenvalues of $A$ are distinct and $A$ has the SSPWL.
\end{theorem}
\begin{proof}
Let $\nabla s_k = \nabla s_k(A)$.  We consider two cases.  Firstly, when the eigenvalues of $A$ contain repeated elements, $\{H\nabla s_1H, H\nabla s_2H, \ldots, H\nabla s_{n-1}H\}$ is a linearly dependent set by \cref{prop:gradientprojection}.  Observe that $\inp{\nabla s_i}{E_{e_j}} = \inp{H\nabla s_iH}{E_{e_j}}$ since $E_{e_j}\in\msymz$ and $HE_{e_j} = E_{e_j}H = E_{e_j}$.  Therefore, the rows of $J_A$ are dependent and $J_A$ does not have full row rank.  

Secondly, when the eigenvalues of $A$ are distinct, by \cref{prop:gradientprojection} the set $\{H\nabla s_1H, H\nabla s_2, \ldots, H\nabla s_{n-1}H\}\subset \msymz$ is independent.  Let $F$ be the function as in \cref{eq:perturb}.  Consider a trajectory $A(t) = e^{-K(t)}Ae^{K(t)}$ with $K(0) = O$ and $A(0) = A$.  Since $s_k$ is a constant for any $A(t)$, we have $\inp{\nabla s_k}{\dot{A}(0)} = 0$.  By properly choosing $K(t)$, we may obtain $\dot{A}(0)$ to be any element in $\range(F_K))$, so $\nabla s_k$ is orthogonal to any element in $\range(F_K)$.  Moreover, 
\[
    \inp{H\nabla s_kH}{K\trans A - AK} = \inp{\nabla s_k}{H(K\trans A - AK)H} = \inp{\nabla s_k}{K\trans A - AK} = 0
\]
for any $K\trans A - AK\in\range(F_K)^\perp$ since $K\trans A - AK\in\msymz$.  Thus, $H\nabla s_kH\in\range(F_K)^\perp$ for each $k$.  By \cref{cor:dim}, $\dim(\range(F_K)^\perp) = n - 1$.  Therefore, $\{H\nabla s_1H, H\nabla s_2H, \ldots, H\nabla s_{n-1}H\}$ is a basis of $\range(F_K)^\perp$.  On the other hand, $\{E_{e_j}\}$ for nonzero entries $e_j$ is a basis of $\range(F_B)$.  Recall that $A$ has the SSPWL if and only if $\range(F_B) + \range(F_K) = \msymz$ by \cref{prop:rangestrong}.  The result now follows from \cref{prop:tnequiv}.
\end{proof}

\subsection{Every Connected Graph Realizes a Matrix with SSPWL}

In this subsection, we show that connectedness  is necessary  and sufficient for a graph to admit a matrix with the SSPWL. To prove this, we rely on the equivalence of a matrix having the SSPWL and  having a full rank Jacobian matrix (\cref{thm.SSPWL-Jac}).

\begin{lemma}
\label{lem:Jac_submatrix}
Let $T$ be a  tree on $n$ vertices with variable edge weights $w_1, \dots, w_{n-1}$ such that $w_{n-1}$ is a weight on a pendant edge. Let $T'$ be the tree on $n-1$ vertices  obtained from $T$ by removing the pendant edge and leaf corresponding to the weight $w_{n-1}$. For $A \in \mptn_L(T)$, let
$p_A(x) = \det(A- xI)$ be the characteristic polynomial of $A$ where the coefficient of $(-x)^{n-i}$ is  $s_i(w_1, \dots, w_{n-1})$, the $i$th symmetric function of the eigenvalues of $A$. Let $J_A$ be the Jacobian matrix of $A$ with $i,j$-entry given by $\frac{\partial s_i}{\partial w_j}$ evaluated at the weights $(w_1, \dots, w_{n-1})$ of $A$. Then $J_A$ evaluated at $(w_1, \dots, w_{n-2}, 0)$
is of the form 
\[
\left[\begin{array}{c|c}
J_{A'} & \ast \\ \hline
\mathbf{0} & c
\end{array}
\right],
\]
where $J_A'$ is the Jacobian matrix with respect to the matrix $A' \in \mptn_L(T')$ with edge weights $w_1, \dots, w_{n-2}$ and $c =  nw_1 w_2  \cdots w_{n-2}$.
\end{lemma}
\begin{proof}
For $i, j < n-1$, the $i,j$-entry of $J_A$ is $\frac{\partial s_i(w_1, \dots, w_{n-2}, w_{n-1})}{\partial w_j}$. Note that the value of this expression when $w_{n-1}$ is set equal to 0 is the same as $\frac{\partial s_i(w_1, \dots, w_{n-2}, 0)}{\partial w_j}$; that is, if $g_i(w_1, \dots, w_{n-2}) = s_i(w_1, \dots, w_{n-2}, 0)$, then  $$\frac{\partial g_i(w_1, \dots, w_{n-2})}{\partial w_j} = \left. \frac{\partial s_i(w_1, \dots, w_{n-2}, w_{n-1})}{\partial w_j} \right |_{w_{n-1}= 0}.$$
On the other hand, if $A' \in \mptn_L(T')$ as described in the statement of the Lemma, then $A'$ is the order $n-1$ submatrix of the matrix obtained from $A$ with $w_{n-1} = 0$. Thus, $p_{A'}(x) = \det(A' - xI)$ where the coefficient of $(-x)^{n-1-i}$ is 
$s_i(w_1, \dots, w_{n-2}) = g_i(w_1, \dots, w_{n-2})$. Thus, for $i, j < n-1$, the $i,j$-entry of $J_A$ when evaluated at $(w_1, \dots, w_{n-2}, 0)$ is equal to the $i,j$-entry of $J_A'$. It is left to prove that the last row of $J_A$ when evaluated at $(w_1, \dots, w_{n-2}, 0)$ is $[\ 0 \ \cdots \  0 \ | \ c \ ]$, where $c=nw_1 \cdots w_{n-2}$.  Since the entries in this row are equal to $$\left.\frac{\partial s_{n-1}(w_1, \dots, w_{n-1})}{\partial w_j}\right |_{w_{n-1}= 0}, \ j = 1, \dots, n-1,$$ we are done once we verify that 
\begin{equation}
\label{eq:sn-1}
s_{n-1}(w_1, \dots, w_{n-1}) = nw_1\cdots w_{n-1},
\end{equation}
which follows from the Weighted Matrix-Tree Theorem (see, e.g., \cite{KS}).
\end{proof}

\begin{lemma}
\label{lem:jac_tree}
Let $n \geq 2$. For any tree $T$ on $n$ vertices, there exist $w_1, \dots, w_{n-1} > 0 $ such that $A \in \mptn_L(T)$ with these edge weights satisfies $\det J_A \neq 0$.
\end{lemma}

\begin{proof}
We prove the result by induction on $n$.  For $n =2$, choose any $w_1 \neq 0$. For $n=3$, $T= P_3$ and any $A \in \mptn_L(T)$ satisfies
\[
A = \begin{bmatrix}
    w_1  & -w_1 & 0 \\
    -w_1 & w_1 + w_2 & -w_2\\
    0 & -w_2 & w_2
\end{bmatrix}, \ 
J_A = \begin{bmatrix}
2 & 2 \\
3w_2 & 3w_1 
\end{bmatrix}.
\]
Thus, $\det J_A \neq 0$ for any $w_1, w_2$ with $w_1 \neq w_2$.

For the purpose of an induction argument, assume that the result holds for all trees on fewer than $n$ vertices. Suppose $T$ is a  tree on $n$ vertices with variable edge weights $w_1, \dots, w_{n-1}$ such that $w_{n-1}$ is a weight on a pendant edge, and  $T'$ is the tree on $n-1$ vertices   with weights $w_1, \dots, w_{n-2}$, that is, $T'$ is obtained from $T$ by removing the pendant edge and leaf. Let $A \in \mptn_L(T)$ and denote by $A_0$ the matrix obtained from $A$ by setting $w_{n-1} = 0$. By \cref{lem:Jac_submatrix}, 
\[
J_{A_0}  =  \left[\begin{array}{c|c}
J_{A'} & \ast \\ \hline
\mathbf{0} & c
\end{array}
\right],
\]
where $J_{A'}$ is the Jacobian matrix with respect to the matrix $A' \in \mptn_L(T')$ with edge weights $w_1, \dots, w_{n-2}$ and $c =  nw_1 w_2  \cdots w_{n-2}$. Using the induction hypothesis, there exist positive weights $w_1, \dots, w_{n-2} > 0$ such that $\det J_{A'}\neq 0$. Hence, $\det J_{A_0} = \det J_{A'} \cdot c \neq 0$. With this choice of weights $w_1, \dots, w_{n-2}$, consider $\det J_A$ as a polynomial in $w_{n-1}$, and call this polynomial $q(w_{n-1})$. We have $q(0) \neq 0$. Since $q$ is a continuous function of $w_{n-1}$, there exists $\varepsilon > 0$ such that $q(\varepsilon) \neq 0$. Then the matrix $A_\varepsilon \in \mptn_L(T)$ with weights $w_1, \dots, w_{n-2}, w_{n-1} = \varepsilon$ has $\det J_{A_\varepsilon} \neq 0$.
\end{proof}

\begin{theorem}
    \label{thm:}
   The set  $\mptn_L(G)$ admits a  matrix with SSPWL if and only if $G$ is connected.
\end{theorem}

\begin{proof}
Suppose $G$ is not connected.  Assume $A\in\mptn_L(G)$ has the form $A = \begin{bmatrix}
    B & O \\O & C
\end{bmatrix}$.  By selecting $X = \begin{bmatrix} O & J\trans \\ J & O \end{bmatrix}$, where $J$ is the all-ones matrix of appropriate size, it follows that $A\circ X = I\circ X = O$ and $H(AX - XA)H = O$.  Therefore, every matrix in $\mptn_L(G)$ loses the SSPWL if $G$ is disconnected.

For sufficiency, assume that $G$ is connected. Then $G$ contains a spanning tree $T$. By \cref{lem:jac_tree}, there exists $A \in \mptn_L(T)$ such that $\det J_A \neq 0$. Hence, by \cref{thm.SSPWL-Jac}, $A$ has the SSPWL. The result now follows from the Supergraph lemma \cref{lem:supergraph}. 
\end{proof}

The next result demonstrates that starting with an SSPWL matrix whose graph is a tree, appending an edge results in an SSPWL matrix  for all but a finite number of weights that are assigned to the new edge. 

\begin{corollary}   
Let $n \geq 2$ and let $T$ be a tree on $n$ vertices. For $A \in \mptn_L(T)$, $w > 0$ and $j \in \{1, \dots, n\}$, define
    \[
    \widehat{A} =  \left[\begin{array}{c|c}
A + w\mathbf{e}_j\mathbf{e}_j\trans & -w\mathbf{e}_j\ \\ \hline
-w\mathbf{e}_j\trans & w
\end{array}
\right] \in \mptn_L(\widehat{T}),
    \]
    where $V(\widehat{T}) = V(T) \cup \{n+1\}$ and $E(\widehat{T}) = E(T) \cup \{j, n+1\}$. If $A$ has the SSPWL, then  $\widehat{A}$ has the SSPWL  for all but a finite number of values of $w >0$.
    \end{corollary}

\begin{proof}
Let $w_1, \dots, w_{n-1}$ be the weights on the edges of $T$ and consider $\det J_{\widehat{A}}$ as a polynomial in $w$ (treating $w_1, \dots, w_{n-1}$ as constants), call this polynomial $q(w)$. By \cref{lem:Jac_submatrix}, $q(0) = \det J_A \cdot nw_1\dots w_{n-1}$. Since $A \in \mptn_L(T)$ has the SSPWL, the eigenvalues of $A$ are distinct by \cref{cor:tree-SSPWL} and thus  $\det J_A \neq 0$ by \cref{thm.SSPWL-Jac}. Thus $q(0) \neq 0$,  and since $q$ is a  polynomial of degree at most $n$, $q(w) \neq 0$, except for at most $n$ values of $w$ (see also the proof of Lemma \ref{lem:jac_tree}). Thus, $\det J_{\widehat{A}} \neq 0$ except for at most $n$ values of $w$.
\end{proof}

\section{Stars, \texorpdfstring{$K_n - e$}{Kn - e}, and Graphs on Four Vertices}
\label{sec:iepl4}

In this section, we apply the analysis on strong weighted Laplacian matrices and the corresponding Jacobian Method from the previous sections to two specific graph families and carefully review (relative to some of the work in \cite{IEPL})
the spectral regions in $S_{L}$  for connected graphs on 4 vertices.  For short, we call a matrix that has the SSPWL as a \emph{strong} matrix; otherwise, it is a \emph{weak} matrix.

\subsection{Stars}
 In this subsection, we focus on strong weighted Laplacian matrices whose graph is a star, namely $K_{1,n-1}$. 
Let
\[ A= \left[\begin{array}{c|ccc}
\sum_{i} w_i & -w_1 & \cdots & -w_{n-1} \\
\hline
-w_1 & w_1 & ~ & 0 \\
\vdots & ~ & \ddots & ~ \\
-w_{n-1} & 0 & ~ & w_{n-1}\\
\end{array} \right]
\]
be a weighted Laplacian matrix that fits the star. Then a basic computation reveals that the sum of all $k\times k$ minors of $A$ can be written as
\[ s_{k} = (k+1) \sum _{|\alpha|=k} \prod_{i \in \alpha} w_i, \]
where $\alpha$ is a $k$-subset of $\{1,2,\ldots, n-1\}$ for $k = 1,2,\ldots, n-1$.

Given a vector $\bw =(w_1, w_2, \ldots, w_{n-1})\trans$, we let 
$t_k(\bw)$ denote the $k$th elementary symmetric function of the entries of $\bw$, for $k=1,2,\ldots,n-1$. Then it follows that $s_k = (k+1)t_k(\bw)$. Furthermore, if $i_1, i_2, \ldots, i_j$ is a $j$-subset of $\{1,2,\ldots, n-1\}$, then the notation $\bw^{(i_1, i_2, \ldots, i_j)}$ denotes the sub-vector obtained from $\bw$ by removing the entries $w_{i_1}, w_{i_2}, \ldots, w_{i_j}$. 

Then, basic differentiation implies that 
\[ 
    \frac{\partial s_k}{\partial w_i} = (k+1) t_{k-1}(\bw^{(i)}),
\] 
where $k=1,2,\ldots,n-1$ and $i=1,2,\ldots,n-1$ (note: $t_0 :=1$). In this case, it follows that the Jacobian matrix with respect to the matrix $A \in \mptn_L(K_{1,n-1})$ with edge weights $w_1, \dots, w_{n-1}$ is the $(n-1)\times (n-1)$ matrix equal to
\[ J_{A}= \left[\begin{array}{cccc}
2 & 2 & \cdots & 2 \\
3 t_{1}(\bw^{(1)}) & 3 t_{1}(\bw^{(2)}) & \ldots & 3 t_{1}(\bw^{(n-1)}) \\
4 t_{2}(\bw^{(1)}) & 4 t_{3}(\bw^{(2)}) & \ldots & 4 t_{2}(\bw^{(n-1)}) \\
\vdots & \vdots & \cdots & \vdots \\
n t_{n-2}(\bw^{(1)}) & n t_{n-2}(\bw^{(2)}) & \ldots & n t_{n-2}(\bw^{(n-1)}) \\
\end{array} \right]. \]
To determine a criterion regarding the invertibility of $J_A$, we consider $\det (J_A)$. In this case applying elementary column operations to eliminate the 2's in row 1, it follows that $\det(J_A)$ is equal to 
\small{\[\hspace{-1cm} 2 \det \left[\begin{array}{cccc}
3(t_{1}(\bw^{(2)})-t_{1}(\bw^{(1)})) & 3(t_{1}(\bw^{(3)})-t_{1}(\bw^{(1)})) & \ldots & 3(t_{1}(\bw^{(n-1)})-t_{1}(\bw^{(1)})) \\
4(t_{2}(\bw^{(2)})-t_{2}(\bw^{(1)})) & 4(t_{2}(\bw^{(3)})-t_{2}(\bw^{(1)})) & \ldots & 4(t_{2}(\bw^{(n-1)})-t_{2}(\bw^{(1)})) \\
\vdots & \vdots & \cdots & \vdots \\
n(t_{n-2}(\bw^{(2)})-t_{n-2}(\bw^{(1)})) & n(t_{n-2}(\bw^{(3)})-t_{n-2}(\bw^{(1)})) & \ldots & n(t_{n-2}(\bw^{(n-1)})-t_{n-2}(\bw^{(1)})) \\
\end{array} \right], \]} where the latter matrix has size $(n-2)\times (n-2)$. In fact, the entries of this latter matrix are of the form 
$(k+2)(t_{k}(\bw^{(j)})-t_{k}(\bw^{(1)}))$, where $k=1,\ldots, n-2$ and $j=2,3,\ldots, n$. A straightforward algebraic simplification involving symmetric functions reveals
\[ t_{k}(\bw^{(j)})-t_{k}(\bw^{(1)}) = (w_1-w_j)t_{k-1}(\bw^{(1,j)}),\] where 
$k=1,\ldots, n-2$ and $j=2,3,\ldots, n$ and with the convention that $t_0 :=1$.
Applying the multi-linearity of the determinant function implies
\[ \det(J_A) = 2 \cdot 
\left(\prod_{j=2}^{n-1} (w_1-w_j)\right) \det(B),\] where the entries of $B$ are of the form $(k+2)t_{k-1}(w^{(1,j)})$. Hence, we can apply similar arguments as above to establish that 
\[ \det(J_A) = (n-1)! \cdot \prod_{i<j} (w_j-w_i).\]
With \cref{cor:tree-SSPWL,thm.SSPWL-Jac}, this formula implies \cref{thm:weakstar}.

\begin{theorem}
\label{thm:weakstar}
Choose the matrix $A \in \mptn_L(K_{1,n-1})$ with edge weights $w_1, \dots, w_{n-1}$. Then
$A$ has the SSPWL if and only if the edge weights $w_1, w_2, \ldots, w_{n-1}$ represent distinct positive numbers.
\end{theorem}

\subsection{The Complete Graph Minus an Edge}
\label{ssec:knme}

The case of a complete graph is simple in that all weighted Laplacian matrices are strong. However, removing a single edge does allow for some weak matrices which are characterized below. To this end, 
let $e = \{1, n\}.$ Suppose that $A = [a_{ij}] \in \mptn_L(K_n - e)$ does not have the SSPWL. Then there exists $c \neq 0$ such that $X = cE_{1n}$ satisfies $H(AX - XA)H = O$,  or equivalently, $H(AE_{1n} - E_{1n}A)H = O$.  Note that $a_{1n} = 0$ since $A\in\mptn_L(K_n - e)$.  By direct computation, $AE_{1n} - E_{1n}A$ is
\[
\begin{array}{cl}
 &  \left[\begin{array}{c|c|c}
0 & 0 \  \cdots  \ 0 & a_{11}\\
\hline
a_{2n} & & a_{21} \\
\vdots & O & \vdots \\
a_{n-1, n} & & a_{n-1,1}\\
\hline
a_{nn} & 0 \ \cdots \ 0 & 0
\end{array}
\right]
-
\left[\begin{array}{c|c|c}
0 & a_{n2} \  \cdots  \ a_{n,n-1} & a_{nn}\\
\hline
0 &  & 0\\
\vdots & O & \vdots \\
0 &  & 0\\
\hline
a_{11} & a_{12} \ \cdots \ a_{1,n-1}& 0
\end{array}
\right] \\
 & \\
 = & 
\left[\begin{array}{c|c|c}
0 & -a_{2n} \  \cdots  \ -a_{n-1,n} & a_{11} - a_{nn}\\
\hline
a_{2n} & & a_{12} \\
\vdots & O & \vdots \\
a_{n-1, n} & & a_{1, n-1}\\
\hline
a_{nn} - a_{11} & -a_{12} \ \cdots \ -a_{1,n-1}& 0
\end{array}
\right].
\end{array}
\]
By \cref{prop:hzero}, $H(AE_{1n} - E_{1n}A)H = O$ is equivalent to $(\be_{i_1} - \be_{j_1})\trans (AE_{1n} - E_{1n}A) (\be_{i_2} - \be_{j_2}) = 0$ for arbitrary $i_1 \neq j_1$ and $i_2 \neq j_2$.  
By taking $(i_1,j_1) = (1,2)$ and $(i_2,j_2) = (1,3)$, we see that $a_{2n} = a_{3n}$.  By switching $j_2$ to $4, \ldots, n-1$, we get $a_{2n} = a_{3n} = \cdots = a_{n-1,n}$.  Similarly, by taking $(i_1,j_1) = (2,n)$ and $(i_2,j_2) = (3,n), \ldots, (n-1,n)$, we have $a_{12} = a_{13} = \cdots = a_{1,n-1}$.  
Since $A\in\mptn_L(K_n - e)$, we know $a_{11} = -a_{12} - \cdots - a_{1,n-1} = -(n-2)a_{12}$ and $a_{nn} = -a_{2n} - \cdots - a_{n-1,n} = -(n-2)a_{2n}$.  Finally, we consider $(i_1,j_1) = (1,2)$ and $(i_2,j_2) = (1,n)$ and get 
\[
    a_{12} - a_{2n} - (a_{11} - a_{nn}) = (n-1)(a_{12} - a_{2n}) = 0,
\]
which implies $a_{12} = a_{2n}$.  

In summary, $A\in\mptn_L(K_n - e)$ is a weak matrix if and only if 
\[
    a_{2n} = \cdots = a_{n-1,n} = a_{12} = \cdots = a_{1,n-1} = -\alpha
\]
for some $\alpha > 0$. These computations and  \cite[Lemma~2.4]{IEPL} prove Theorem \ref{thm:K-e}.  

\begin{theorem}
\label{thm:K-e}
Let $e=\{1,n\}$ and $A \in \mathcal{S}_L(K_n - e)$. Then $A$ does not have the SSPWL if and only if it is of the form 
\[
A = \left[\begin{array}{c|c|c}
(n-2)\alpha & -\alpha \bone_{n-2}\trans & 0\\
\hline
-\alpha \bone_{n-2} & B + 2\alpha I_{n-2} & - \alpha \bone_{n-2}\\
\hline
 0 & - \alpha \bone_{n-2}\trans & (n-2) \alpha \rule{0pt}{2.5ex}
\end{array}
\right]
\]
where $B\in \mptn_L(K_{n-2})$ and $\alpha > 0$. Furthermore,
\begin{equation}
\label{alpha.spec}
\spec(A) = \{0, (n-2)\alpha, n \alpha, \mu_2 + 2\alpha, \dots, \mu_{n-2} +2\alpha \}
\end{equation}
where $\{0, \mu_2, \dots, \mu_{n-2}\} = \spec(B)$. 
\end{theorem}

\subsection{Connected Graphs of Order Four}

In this subsection, we characterize all weak matrices for connected graphs of order four and use them to find the potential boundaries of the spectra.  To be more precise, for each connected graph of order four $G$, we study the spectra of matrices $A\in\mptn_L(G)$ under the normalization that $\tr(A) = 2m$, where $m = |E(G)|$.  Thus, we may assume $\lambda_2 + \lambda_3 + \lambda_4 = 2m$ and plot all spectra on the $\lambda_2,\lambda_3$-plane.  Each blue dot $(\lambda_2, \lambda_3)$ in \cref{fig:g4-weak} represents a spectrum $\{\lambda_2, \lambda_3, 2m - \lambda_2 - \lambda_3\}$ given by some matrix $A\in\mptn_L(G)$ with $\tr(A) = 2m$.  They are generated from the same simulation process as in \cite{IEPL}.  The spectra of weak matrices lie on the orange curves.  By the Bifurcation lemma (\cref{lem:bifurcation}), if a matrix has the SSPWL, then all nearby spectra are realizable as well.  That is, the boundaries of the blue dots only possibly occur on the orange curves, other than the trivial boundaries $\lambda_2 = 0$, $\lambda_2 = \lambda_3$, and $\lambda_3 = \lambda_4$ (equivalently, $2\lambda_3 + \lambda_4 = 2m$).

\begin{figure}[h]
\centering
\includegraphics[width=\linewidth]{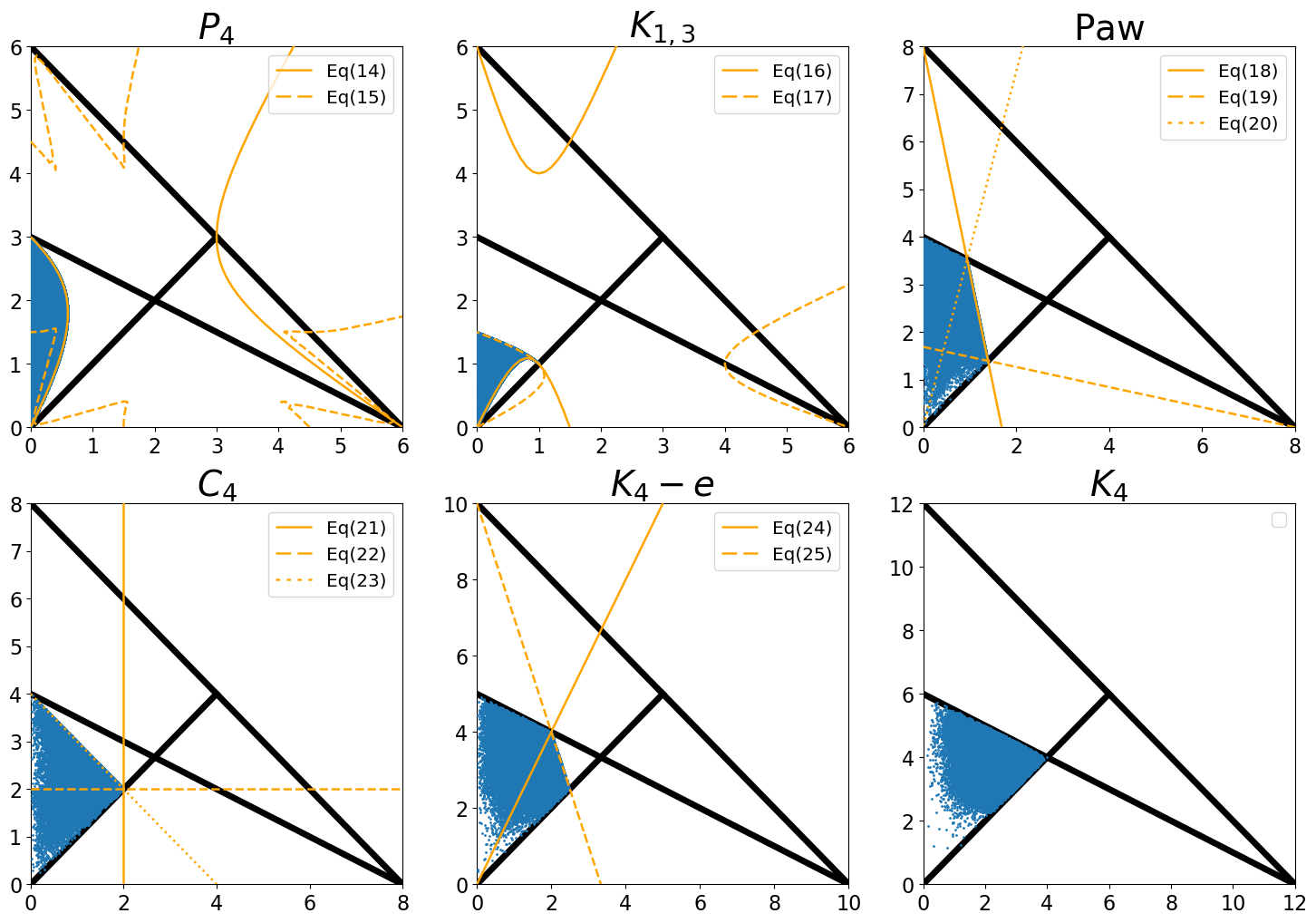}
\caption{Potential spectra for each graphs and potential boundaries led by weak matrices, drawn on the $\lambda_2,\lambda_3$-plane.}
\label{fig:g4-weak}
\end{figure}

Here we describe a general scheme to study the weak matrices of a given graph $G$ and to obtain the equations for the orange curves.  We assume $G$ has $n$ vertices and $m$ edges.
\begin{itemize}
\item Consider a variable matrix $A\in\mptn_L(G)$ and calculate its verification matrix $\Psi$.  Whenever necessary, we normalize the matrix $A$ into $\frac{2m}{\tr(A)}A$ such that $\tr(A) = 2m$.  
\item Compute the maximal minors of $\Psi$, that is, the determinant of the maximal square submatrices.  Thus, $A$ loses the SSPWL if and only if these maximal minors are zero.  
\item Compute the sum of all $k\times k$ minors of $A$ as $s_k$ and then find equations of $s_k$'s implied by the maximal minors.  Let $\mathcal{I}$ be the ideal generated by the maximal minors and the relations between $s_k$'s and the weights.  Through the elimination ideal of $\mathcal{I}$ by removing all weight variables, one may find the equations of $s_k$'s.  Recall that we may normalize $s_1 = 2m$.
\item Based on Vieta's formulas, substitute $s_k$'s with the eigenvalues $\lambda_i$'s in the previously found equations of $s_k$'s.  Recall that we may normalize $\lambda_4 = 2m - \lambda_2 - \lambda_3$.
\end{itemize}

In the following cases, we use the same $P$ matrix as in \cref{ex:p4sspwl} for generating the verification matrix $\Psi$.

\begin{figure}[h]
\centering
\begin{center}
\begin{tikzpicture}
\node[label={below:$1$}] (1) at (0,0) {};
\node[label={below:$2$}] (2) at (1,0) {};
\node[label={below:$3$}] (3) at (2,0) {};
\node[label={below:$4$}] (4) at (3,0) {};
\draw (1) -- node[midway,above,rectangle,draw=none]{$x$}
(2) -- node[midway,above,rectangle,draw=none]{$z$}
(3) -- node[midway,above,rectangle,draw=none]{$y$}
(4);
\node[rectangle,draw=none] (name) at (1.5,-1.5) {$P_4$};
\end{tikzpicture}
\hfil
\begin{tikzpicture}
\node[label={below:$1$}] (1) at (-1,0) {};
\node[label={below:$2$}] (2) at (0,0) {};
\node[label={above:$3$}] (3) at (30:1) {};
\node[label={below:$4$}] (4) at (-30:1) {};
\draw (1) -- node[midway,above,rectangle,draw=none]{$x$}
(2) -- node[midway,above,rectangle,draw=none]{$y$}
(3) -- node[midway,right,rectangle,draw=none]{$z$}
(4) -- node[midway,above,rectangle,draw=none]{$w$}
(2);
\node[rectangle,draw=none] (name) at (0,-1.5) {$\Paw$};
\end{tikzpicture}
\hfil
\begin{tikzpicture}
\foreach \i in {1,2,3,4} {
    \pgfmathsetmacro{\ang}{135 + 90*(\i - 1)}
    \node[label={\ang:$\i$}] (\i) at (\ang:0.707) {};
}
\draw (1) -- node[midway,left,rectangle,draw=none]{$x$}
(2) -- node[midway,below,rectangle,draw=none]{$y$}
(3) -- node[midway,right,rectangle,draw=none]{$z$}
(4) -- node[midway,above,rectangle,draw=none]{$w$}
(1);
\node[rectangle,draw=none] (name) at (0,-1.5) {$C_4$};
\end{tikzpicture}
\end{center}
\caption{Some weighted graphs of order four.}
\label{fig:wg4}
\end{figure}
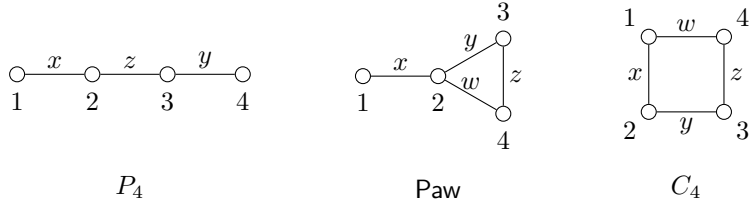

\subsubsection{Path Graph \texorpdfstring{$P_4$}{P4}}

Consider the weighted graph $P_4$ with weights $x,z,y$ as shown in \cref{fig:wg4}.  Let $A\in\mptn_L(P_4)$ be the variable matrix and $\Psi$ its verification matrix.  Then
\[
    \Psi = \begin{bmatrix}
        -2x + y + 2z & -y + z & x - 2y - z \\
        -y & -2x + y & -x + 2y \\
        y & 2x - y + z & x - z
    \end{bmatrix},
\]
with determinant equal to
\[
    \det(\Psi) = 2(x - y)(4xy - 3xz - 3yz + 3z^2).
\]
Thus, $A$ does not have the SSPWL if and only if 
\begin{align}
    x - y &= 0, \text{ or } \label{eq:w_out} \\
    4xy - 3xz - 3yz + 3z^2 &= 0. \label{eq:w_in}
\end{align}

Let $\spec(A) = \{0, \lambda_2, \lambda_3, \lambda_4\}$ with $\lambda_2 < \lambda_3 < \lambda_4$.  Let $s_k$ be the sum of all $k\times k$ principal minors of $A$.  By direct computation, we have 
\begin{equation}
\label{eq:swlam}
    \begin{array}{rll}
        s_1 &= 2x + 2y + 2z &= \lambda_2 + \lambda_3 + \lambda_4, \\
        s_2 &= 3xz + 3yz + 4xy &= \lambda_2\lambda_3 + \lambda_2\lambda_4 + \lambda_3\lambda_4, \\
        s_3 &= 4xyz &= \lambda_2\lambda_3\lambda_4.
    \end{array}
\end{equation}

In the following, we always normalize the matrix $A$ by $\frac{2m}{\tr(A)}A$ and assume $\tr(A) = s_1 = 2m = 6$.  

Under this normalization with $z = 3 - x - y$, \cref{eq:w_out,eq:w_in} becomes 
\begin{align}
    x - y &= 0, \text{ or } \label{eq:w_out_n} \\
    6x^2 + 16xy + 6y^2 - 27x - 27y + 27 &= 0. \label{eq:w_in_n}
\end{align}
Thus, we observe that \cref{eq:w_out_n,eq:w_in_n} led to 
\begin{align}
    s_2^3 - 9s_2^2 - 30s_2s_3 + 8s_3^2 + 243s_3 &= 0, \text{ or } \label{eq:s_out_n} \\
    4s_2^3 - 27s_2^2 - 162s_2s_3 + 81s_3^2 + 972s_3 &= 0. \label{eq:s_in_n}
\end{align}
To see these, we substitute $s_k$'s with $x,z,y$ using \cref{eq:swlam} and normalize with $z = 3 - x - y$.  Then \cref{eq:s_in_n} is equivalent to 
\begin{equation}
\label{eq:rel_out_n}
\begin{aligned}
    (x - y)^2  (27x^4 + 108x^3y - 243x^3 + 178x^2y^2 - 693x^2y + 810x^2 & \\
    + 108xy^3 - 693xy^2 + 1512xy - 1215x & \\
    + 27y^4 - 243y^3 + 810y^2 - 1215y + 729) &= 0.
\end{aligned}
\end{equation}
Also, \cref{eq:s_out_n} is equivalent to
\begin{equation}
\label{eq:rel_in_n}
    (-3x + y)  (-x + 3y)  (6x^2 + 16xy - 27x + 6y^2 - 27y + 27)^2 = 0.    
\end{equation}

\begin{remark}
Behind the scenes, \cref{eq:s_out_n,eq:s_in_n} are obtained via the elimination ideal.  Let $\mathcal{I}$ be the ideal generated by \cref{eq:w_out_n}, the three relations between $s_i$'s and $x,z,y$ in \cref{eq:swlam}, and the normalization $s_1 = 6$.  The elimination ideal of $\mathcal{I}$ by removing $x,y,z$ and $s_1$ is generated by \cref{eq:s_out_n}.  Similarly, if we replace \cref{eq:w_out} with \cref{eq:w_in} in the generators of $\mathcal{I}$, then the elimination ideal is generated by \cref{eq:s_in_n}.
\end{remark}

After replacing $s_k$'s with $\lambda_i$'s using \cref{eq:swlam} and normalizing with $\lambda_4 = 6 - \lambda_2 - \lambda_3$, \cref{eq:s_out_n} leads to
\begin{equation}
\label{eq:lam_out_all}
\begin{aligned}
    (-\lambda_2^2 - \lambda_2\lambda_3 + 6\lambda_2 + \lambda_3^2 - 3\lambda_3)
    (-\lambda_2^2 + \lambda_2\lambda_3 + 3\lambda_2 + \lambda_3^2 - 6\lambda_3) & \\
    (\lambda_2^2 + 3\lambda_2\lambda_3 - 9\lambda_2 + \lambda_3^2 - 9\lambda_3 + 18) &= 0.
\end{aligned}
\end{equation}
In terms of $0 < \lambda_2 \leq \lambda_3 \leq \lambda_4$ and $\lambda_2 + \lambda_3 + \lambda_4 = 6$, only 
\begin{equation}
\label{eq:lam_out}
\begin{split}
    -\lambda_2^2 - \lambda_2\lambda_3 + 6\lambda_2 + \lambda_3^2 - 3\lambda_3 = 0
\end{split}
\end{equation}
is effective.

Under the same replacement and normalization, \cref{eq:s_in_n} leads to 
\begin{equation}
\label{eq:lam_in}
\begin{aligned}
    4\lambda_2^6 + 12\lambda_2^5\lambda_3 - 72\lambda_2^5 - 57\lambda_2^4\lambda_3^2 - 54\lambda_2^4\lambda_3 + 459\lambda_2^4 & \\
    - 134\lambda_2^3\lambda_3^3 + 936\lambda_2^3\lambda_3^2 - 594\lambda_2^3\lambda_3 - 1188\lambda_2^3 & \\
    - 57\lambda_2^2\lambda_3^4 + 936\lambda_2^2\lambda_3^3 - 4023\lambda_2^2\lambda_3^2 + 3564\lambda_2^2\lambda_3 + 972\lambda_2^2 & \\
    + 12\lambda_2\lambda_3^5 - 54\lambda_2\lambda_3^4 - 594\lambda_2\lambda_3^3 + 3564\lambda_2\lambda_3^2 - 3888\lambda_2\lambda_3 & \\
    + 4\lambda_3^6 - 72\lambda_3^5 + 459\lambda_3^4 - 1188\lambda_3^3 + 972\lambda_3^2 &= 0.
\end{aligned}
\end{equation}

\cref{eq:lam_out,eq:lam_in} are presented in \cref{fig:g4-weak} as the orange curves for $P_4$.

\subsubsection{Star Graph \texorpdfstring{$K_{1,3}$}{K13}}

Consider the weighted graph $K_{1,3}$ with weights $x,y,z$.  Let $A\in\mptn_L(K_{1,3})$ be the variable matrix.  By \cref{thm:weakstar}, $A$ is a weak matrix if and only if $x - y = 0$, $x - z = 0$, or $y - z = 0$.  Each of these relations imply
\[
    64 s_{2}^{3} - 432 s_{2}^{2} - 1944 s_{2} s_{3} + 729 s_{3}^{2} + 11664 s_{3} = 0,
\]
which is equivalent to
\[
    (x - y)^2 (x - z)^2 (y - z)^2 = 0
\]
when $z = 3 - x - y$.  

By substituting $s_k$'s with $\lambda_i$'s and normalizing with $\lambda_4 = 6 - \lambda_2 - \lambda_3$, we obtain the equation 
\begin{align*}
(8\lambda_2^2 + 17\lambda_2\lambda_3 + 8\lambda_3^2 - 84\lambda_2 - 84\lambda_3 + 216)&(8\lambda_2^2 - \lambda_2\lambda_3 - \lambda_3^2 - 12\lambda_2 + 6\lambda_3)\\
&(\lambda_2^2 + \lambda_2\lambda_3 - 8\lambda_3^2 - 6\lambda_2 + 12\lambda_3) = 0.
\end{align*}
In terms of $0 < \lambda_2 \leq \lambda_3 \leq \lambda_4$ and $\lambda_2 + \lambda_3 + \lambda_4 = 6$, only
\begin{align}
    8\lambda_2^2 - \lambda_2\lambda_3 - \lambda_3^2 - 12\lambda_2 + 6\lambda_3 &= 0 \text{ or} \label{eq:star-down} \\
    \lambda_2^2 + \lambda_2\lambda_3 - 8\lambda_3^2 - 6\lambda_2 + 12\lambda_3 &= 0 \label{eq:star-up}
\end{align}
are effective.

\cref{eq:star-down,eq:star-up} are presented in \cref{fig:g4-weak} as the orange curves for $K_{1,3}$.

\subsubsection{Paw Graph \texorpdfstring{$\Paw$}{Paw}}

Consider the weighted graph $\Paw$ with weights $x,y,z,w$ as shown in \cref{fig:wg4}.  Let $A\in\mptn_L(\Paw)$ be the variable matrix and $\Psi$ its verification matrix.  Then
\[
    \Psi = \begin{bmatrix}
        -2x + 2y + z & y - z & x - y - 2z \\
        w - z & 2w - 2x + z & w - x + 2z
    \end{bmatrix}.
\]
By computing its maximal minors, $\Psi$ is not of full row rank if and only if
\[
    \begin{aligned}
        -4w x + 4x^{2} + 3w y - 4x y + 3w z - 4x z + 3y z &= 0, \\
        -3w x + 2x^{2} + 3w y - 2x y + 3w z - 4x z + 3y z &= 0, \text{ and}\\
-2w x + 2x^{2} + 3w y - 3x y + 3w z - 4x z + 3y z &= 0.
    \end{aligned}
\]
Note that these three equations are joined by ``and'', in contrast to the equations for $P_4$ (and $K_{1,3}$) are joined by ``or''.  Equivalently, $A$ is a weak matrix if and only if one of the three equations are satisfied.  These three equations imply 
\[
    6 s_{2}^{3} - 64 s_{2}^{2} - 320 s_{2} s_{3} + 125 s_{3}^{2} + 3072 s_{3} = 0.
\]

By substituting $s_k$'s with $\lambda_i$'s and normalizing with $\lambda_4 = 8 - \lambda_2 - \lambda_3$, we obtain an equation  
\begin{align*}
(6\lambda_2^2 + 6\lambda_2\lambda_3 + \lambda_3^2 - 48\lambda_2 - 16\lambda_3 + 64) & \\
(\lambda_2^2 + 6\lambda_2\lambda_3 + 6\lambda_3^2 - 16\lambda_2 - 48\lambda_3 + 64) & \\ 
(\lambda_2^2 - 4\lambda_2\lambda_3 + \lambda_3^2) &= 0,
\end{align*}
which can be factored further as 
\[
    \begin{aligned}
        (\lambda_2 + (\frac{1}{2} - \frac{\sqrt{3}}{6})\lambda_3 + (-4 + \frac{4\sqrt{3}}{3}))
(\lambda_2 + (\frac{1}{2} + \frac{\sqrt{3}}{6})\lambda_3 + (-4 - \frac{4\sqrt{3}}{3})) & \\
(\lambda_2 + (3 - \sqrt{3})\lambda_3 - 8)
(\lambda_2 + (3 + \sqrt{3})\lambda_3 - 8) & \\
(\lambda_2 + (-2 + \sqrt{3})\lambda_3)
(\lambda_2 + (-2 - \sqrt{3})\lambda_3) &= 0.
    \end{aligned}
\]
In terms of $0 < \lambda_2 \leq \lambda_3 \leq \lambda_4$ and $\lambda_2 + \lambda_3 + \lambda_4 = 8$, only 
\begin{align}
    \lambda_2 + (\frac{1}{2} - \frac{\sqrt{3}}{6})\lambda_3 + (-4 + \frac{4\sqrt{3}}{3}) &= 0, \label{eq:paw-right} \\
    \lambda_2 + (3 + \sqrt{3})\lambda_3 - 8 &= 0, \text{ or} \label{eq:paw-down} \\
    (\lambda_2 + (-2 + \sqrt{3})\lambda_3) &= 0 \label{eq:paw-up}
\end{align}
is effective.

\cref{eq:paw-right,eq:paw-down,eq:paw-up} are presented in \cref{fig:g4-weak} as the orange curves for $\Paw$.

\subsubsection{Cycle Graph \texorpdfstring{$C_4$}{C4}}

Consider the weighted graph $C_4$ with weights $x,y,z,w$ as shown in \cref{fig:wg4}.  Let $A\in\mptn_L(C_4)$ be the variable matrix and $\Psi$ its verification matrix.  Then
\[
    \Psi = \begin{bmatrix}
        -w - 2x + 2y + z & y - z & 2w + x - y - 2z \\
        -w + z & -2w + 2x + y - z & x - y
    \end{bmatrix}.
\]
By computing its maximal minors, $\Psi$ is not of full row rank if and only if
\[
    \begin{aligned}
        2w^{2} + 2w x - 4x^{2} - 4w y + 2x y + 2y^{2} - 2w z + 4x z - 2y z &= 0, \\
        2w^{2} - 2x^{2} + 4x y - 2y^{2} - 4w z + 2z^{2} &= 0, \text{ and} \\
        4w^{2} - 2w x - 2x^{2} - 4w y + 2x y - 2w z + 4x z + 2y z - 2z^{2} &= 0.
    \end{aligned}
\]
Using the Gr\"obner basis, the three equations above are equivalent to 
\[
    \begin{aligned}
        (w - y)(w - z)(y - z) &= 0, \\
        (w + x - y - z)(w - x - y + z) &= 0, \\
        (w - x - y + z)(w - z) &= 0, \text{ and} \\
        (w + x - y - z)(y - z) &= 0.
    \end{aligned}
\]
Therefore, $A$ is a weak matrix if and only if one of the following hold:
\begin{itemize}
    \item $w = x$ and $y = z$.
    \item $w = y$ and $x = z$, or 
    \item $w = z$ and $x = y$. 
\end{itemize}
These relations on weights lead to the equation  
\[
    8 s_{2}^{2} - 6 s_{2} s_{3} + s_{3}^{2} - 224 s_{2} + 88 s_{3} + 1536 = 0.
\]

In \cref{fig:g4-weak}, $C_4$, the equations of the three colored curves are as follows:
\[
    (\lambda_2 - 2)(\lambda_2 - 4)(\lambda_3 - 2)(\lambda_3 - 4)(\lambda_2 + \lambda_3 - 4)(\lambda_2 + \lambda_3 - 6) = 0.
\]
In terms of $\lambda_2 \leq \lambda_3 \leq \lambda_4$ and $\lambda_2 + \lambda_3 + \lambda_4 = 8$, only 
\begin{align}
    \lambda_2 - 2 &= 0 \label{eq:c4-v}, \\
    \lambda_3 - 2 &= 0, \text{ or} \label{eq:c4-h}\\
    \lambda_2 + \lambda_3 - 4 &= 0 \label{eq:c4-d}
\end{align}
is effective.

\cref{eq:c4-v,eq:c4-h,eq:c4-d} are presented in \cref{fig:g4-weak} as the orange curves for $C_4$.

\subsubsection{\texorpdfstring{$K_4$}{K4} Minus an Edge Graph \texorpdfstring{$K_4 - e$}{K4 - e}}

\cref{ssec:knme} characterized the weak matrices in $\mptn_L(K_4 - e)$.  With $n = 4$, their spectra have the form 
\[
    \{0, 2\alpha, 4\alpha, \mu_2 + 2\alpha\} \text{ with }\alpha, \mu_2 > 0,
\]
so a spectrum $\{\lambda_1, \lambda_2, \lambda_3, \lambda_4\}$ with $\lambda_1 = 0$ and with the normalization $\lambda_2 + \lambda_3 + \lambda_4 = 10$ can be realized by a weak matrix in $\mptn_L(K_4 - e)$ if and only if 
\begin{align}
    \lambda_3 &= 2\lambda_2 \text{ or} \label{eq:k4e-left} \\ 
    \lambda_4 &= 2\lambda_2 \text{ (equivalently, }3\lambda_2 + \lambda_3 = 10\text{)}. \label{eq:k4e-right}
\end{align}

\cref{eq:k4e-left,eq:k4e-right} are presented in \cref{fig:g4-weak} as the orange curves for $C_4$.

\subsubsection{Complete Graph \texorpdfstring{$K_4$}{K4}}

Every matrix in $\mptn_L(K_4)$ has the SSPWL by definition, so no additional potential boundaries are generated.  

\section{Jacobian Methods and Other Strong Properties}
\label{sec:jacall}

As mentioned earlier in \cref{sec:jac}, the Jacobian Method was introduced as a tool to study spectrally arbitrary sign patterns.  This method was then extended to the Nilpotent-Centralizer Method and the non-symmetric strong spectral property (nSSP).  In this section, we make a comprehensive study on the Jacobian Methods under different settings.

\subsection{Jacobian Method and the nSSP}

A \emph{sign pattern} $P$ is a matrix whose entries are in $\{+,-,0\}$.  The \emph{quantitative class} of $P$, denoted by $\mathcal{Q}(P)$, is the set of all matrices whose entries follow the signs of the corresponding entries in $P$.  Note that the arguments in this subsection also hold for zero-nonzero patterns, while we only mention the sign patterns.  Let $A\in Q(\mathcal{P}) \subseteq \mat$.  According to \cite[Section~5]{bifur} the nSSP is equivalent to 
\[T_1 + T_2 = W,\]
where
\begin{itemize}
    \item $T_1 = \vspan(\mathcal{Q}(P))$, 
    \item $T_2 = \{-LA + AL: L\in\mat\}$, 
    \item $W = \mat$.
\end{itemize}

The Jacobian Method under this setting relies on the  Jacobian matrix
\begin{align*}
    J_A = \begin{bmatrix}
        \displaystyle{\frac{d(s_1,\ldots, s_n)}{d({\rm free\ entries})}}
    \end{bmatrix}
\end{align*}
and is a powerful tool in certifying a spectrally arbitrary sign pattern \cite{DJOvdD00,GB13,bifur}.
Let $e_j$ be the indices of the nonzero entries and $E_{e_j}$ the $0$-$1$ matrix whose entry at $e_j$ is $1$ and $0$ otherwise.  Then, by the chain rule, 
\[
    J_A = \begin{bmatrix}
        \inp{\nabla s_i}{E_{e_j}}
    \end{bmatrix}.
\]

\begin{theorem}
\label{thm.Jac-nSSP}
Let $A\in Q(\mathcal{P}) \subseteq \mat$ and $J_A$ its associated Jacobian matrix. Then $J_A$ has full row rank if and only if the minimal polynomial of $A$ has degree $n$ and $A$ has the nSSP.
\end{theorem}
\begin{proof}
We consider two cases.  Firstly, when the minimal polynomial of $A$ has degree less than $n$, the dimension of the span of $\{I, A\trans, \ldots, (A\trans)^{n-1}\}$ is less than $n$, so $\{\nabla s_1, \ldots, \nabla s_n\}$ is linearly dependent by \cref{lem:gradientind}.  Therefore, the rows of $J_A$ are dependent and $J_A$ does not have full row rank.  

Secondly, when the minimal polynomial of $A$ has degree $n$, $\{\nabla s_1, \ldots, \nabla s_n\}$ is linearly independent by \cref{lem:gradientind}.  Also, $T_2$ is known to have codimension $\dim(T_2^\perp) = n$.  Therefore, $\{\nabla s_1, \ldots, \nabla s_n\}$ is a basis of $N_2 = T_2^\perp$.  On the other hand, $\{E_{e_j}\}$ for nonzero entries $e_j$ is a basis of $T_1$.  The result now follows from \cref{prop:tnequiv}.
\end{proof}

\begin{remark}
\cref{thm.Jac-nSSP} explains why the Nilpotent-Jacobian Method in \cite{DJOvdD00}, used  for showing that a pattern is spectrally arbitrary, requires a nilpotent matrix of index $n$.  The paper \cite{GB13} which introduced the Nilpotent-Centralizer method  assumes the same condition.  However,  as explained in \cite{bifur},  nilpotency of index $n$ is not required with the nSSP condition to verify that a pattern is spectrally arbitrary.
\end{remark}

\subsection{Jacobian Method and the SSP}

Let $A\in\mptn(G)\subseteq\msym$.  According to \cite[Proposition~1.1]{bifur} the SSP is equivalent to 
\[T_1 + T_2 = W,\]
where
\begin{itemize}
    \item $T_1 = \vspan(\mptn(G))$, 
    \item $T_2 = \{K\trans A + AK: K\in\mskew\}$, 
    \item $W = \msym$.
\end{itemize}

The Jacobian Method under this setting relies on the Jacobian matrix 
\[
    J_A = \begin{bmatrix}
        \displaystyle{\frac{d(s_1,\ldots, s_n)}{d({\rm free\ entries})}}
    \end{bmatrix},
\]
where the free entries include the diagonal entries and the entries corresponding to edges.  
Let $e_j$ be the indices of the free entries and $E_{e_j}$ the $0$-$1$ matrix with the entry at $e_j$ equal to $1$ and all other entries are $0$.  Note that $E_{e_j}$ has a $1$ when $e_j$ is diagonal and two $1$'s when $e_j$ is off-diagonal.  Then, by the chain rule, 
\[
    J_A = \begin{bmatrix}
        \inp{\nabla s_i}{E_{e_j}}
    \end{bmatrix}.
\]

\begin{theorem}
Let $A\in\mptn(G)\subseteq\msym$ and let $J_A$ be its  associated Jacobian matrix. Then $J_A$ has full rank if and only if the eigenvalues of $A$ are distinct and $A$ has the SSP.
\end{theorem}
\begin{proof}
We consider two cases.  Firstly, when the eigenvalues of $A$ contain repeated elements, the dimension of the span of $\{I, A, \ldots, A^{n-1}\}$ is less than $n$, so $\{\nabla s_1, \ldots, \nabla s_n\}$ is a linearly dependent set by \cref{lem:gradientind}.  Therefore, the rows of $J_A$ are dependent and $J_A$ does not have full row rank.  

Secondly, when the eigenvalues of $A$ are distinct, $\{\nabla s_1, \ldots, \nabla s_n\}$ is a linearly independent set by \cref{lem:gradientind}.  Since $A$ is symmetric, each $\nabla s_k$ is also symmetric. Since $T_2$ is known to have codimension $\dim(T_2^\perp) = n$, with respect to the ambient space $\msym$.  Therefore, $\{\nabla s_1, \ldots, \nabla s_n\}$ is a basis for $N_2 = T_2^\perp$.  On the other hand, $\{E_{e_j}\}$ for nonzero entries $e_j$ is a basis of $T_1$.  The result now follows from \cref{prop:tnequiv}.
\end{proof}

\begin{example}
\label{ex:a3jac}
Consider the matrix 
\[
    A = \begin{bmatrix}
        a & b & c \\
        b & d & e \\
        c & e & f
    \end{bmatrix}.
\]
By direct computation, we have 
\[
    \begin{aligned}
        s_1 &= a + d + f, \\
        s_2 &= ad - b^2 + df - e^2 + af - c^2, \\
        s_3 &= adf + bec + ceb - c^2d - b^2f - ae^2
    \end{aligned}
\]
and the Jacobian matrix $\begin{bmatrix}\frac{d (s_1, \ldots, s_3)}{d (a,b,c,d,e)}\end{bmatrix}$ with respect to every entry is 
\[
    \begin{array}{c|cccccc}
        ~ & a & b & c & d & e & f \\
        \hline
        s_1 & 1 & 0 & 0 & 1 & 0 & 1 \\
        s_2 & d+f & -2b & -2c & a+f & -2e & a+d \\
        s_3 & df-e^2 & 2ec-2bf & 2be-2cd & af-c^2 & 2bc-2ae & ad-b^2
    \end{array}
\]
Note that the columns of the Jacobian matrix corresponding to the off-diagonal entries $b,c,e$ are associated with a factor $2$ since the variable appears in $A$ twice.  Also, note that for a given matrix $A$ with numerical values, $J_A$ is a submatrix of the above mentioned Jacobian matrix obtained by selecting the columns corresponding to the diagonal entries and the entries corresponding to edges. 
\end{example}

\begin{example}
\label{ex:a3gradient}
Using the same matrix $A$ as in \cref{ex:a3jac}, we may compute $(A - xI)\cof$ as 
\[
    \begin{bmatrix}
        (d-x)(f-x)-e^2 & -(b(f-x)-ec) & be-(d-x)c \\
        -(b(f-x)-ec) & (a-x)(f-x)-c^2 & -((a-x)e-bc) \\
        be-(d-x)c & -((a-x)e-bc) & (a-x)(d-x)-b^2
    \end{bmatrix},
\]
which is equal to
{\footnotesize
\[
    (-x)^2\begin{bmatrix}
        1 & 0 & 0 \\
        0 & 1 & 0 \\
        0 & 0 & 1
    \end{bmatrix} +  
    (-x)^1\begin{bmatrix}
        d+f & -b & -c \\
        -b & a+f & -e \\
        -c & -e & a+d
    \end{bmatrix} + 
    (-x)^0\begin{bmatrix}
        df-e^2 & ec-bf & be-cd \\
        ec-bf & af-c^2 & bc-ae \\
        be-cd & bc-ae & ad-b^2
    \end{bmatrix}.
\]
}
The three coefficient matrices are $\nabla s_1, \nabla s_2, \nabla s_3$, respectively, and one may observe that the Jacobian matrix computed in \cref{ex:a3jac} can be obtained from these matrices by recording the entries and multiplying by a factor $2$ for each off-diagonal entry.
\end{example}

\section{Concluding Remarks}
\label{sec:concluding_remarks}

In this paper, we introduce the strong property, SSPWL, for the weighted Laplacian matrix.  This is the first time that a strong property is considered on a specific ambient space $\msymz$, other than the standard spaces $\msym$ or $\mat$.  The Supergraph lemma and the Bifurcation lemma are established in order to study the inverse eigenvalue problem for Laplacian matrices of a graph.  

In parallel with studying the SSPWL, the Jacobian Method is established as an equivalent condition for SSPWL, provided that the eigenvalues are all distinct.  This provides an alternative method for verifying the SSPWL.  For sparse graphs, this method is, in general, easier than checking the SSPWL by definition.  In particular, for trees, the Jacobian Method and the SSPWL are exactly the same, and the Jacobian matrix is of order $n-1$, while the verification matrix of SSPWL is of order $\binom{n-1}{2}$.  

Using these tools, we detect the potential boundaries for the set of spectra given by $\mptn_L(G)$ when $G$ is a connected graph with $4$ vertices.  In \cite{IEPL}, the authors suspected that the boundaries for $\Paw$, $C_4$, $K_4 - e$ are linear, and our computations support this claim.

The study of the potential spectra also leads to the solution of the absolute algebraic connectivity introduced in \cite{Fiedler90}.  The absolute algebraic connectivity of a graph $G$ with $m$ edges, denoted by $\hat{a}(G)$, is defined as the supremum of $\lambda_2$ among matrices $A\in\mptn_L(G)$ with $\tr(A) = 2m$.  That is, it is the right-most points in each picture in \cref{fig:g4-weak}, so the absolute algebraic connectivity is solved for these graphs.  In \cite{Fiedler90}, the absolute algebraic connectivity for trees is considered, and the basic properties of the matrix that achieve the absolute algebraic connectivity as a simple eigenvalue are studied.  With the simulation in \cref{fig:g4-weak} and the boundary formulas we have derived, we see that $\hat{a}(\Paw) = \frac{8}{4 + \sqrt{3}}$, $\hat{a}(C_4) = 2$, and $\hat{a}(K_4 - e) = \frac{5}{2}$, where these graphs are not trees and the absolute algebraic connectivity for each of them is not simple.  More precisely, the combinatorial Laplacian matrix $A\in\mptn_L(C_4)$ is an optimal matrix for $C_4$ with $\lambda_2 = 2$, while $\frac{10}{8}A$ is a limit point of $\mptn_L(K_4 - e)$ and is an optimal matrix for $K_4 - e$ with $\lambda_2 = \frac{5}{2}$.  For $\Paw$, one may trace back through the proof in \cite[Theorem~3.13]{IEPL} and find the unique optimal matrix
\[
    \frac{9 + \sqrt{3}}{78}
    \begin{bmatrix}
        12 & -12 & 0 & 0 \\
        -12 & 28 & -8 & -8 \\
        0 & -8 & 16 -4\sqrt{3} & -8 + 4\sqrt{3} \\
        0 & -8 & -8 + 4\sqrt{3} & 16 -4\sqrt{3}
    \end{bmatrix}
    \in \mptn_L(\Paw)
\]
with $\lambda_2 = \frac{8}{4 + \sqrt{3}}$.  We further point out that the absolute algebraic connectivity has to be achieved by a weak matrix according to the contrapositive statement of the Bifurcation lemma (\cref{lem:bifurcation}).

\section{Acknowledgments}
Thanks to the University of Regina for hosting M.~Catral while on a Faculty Development Research sabbatical from Xavier University, with travel  support from  an AMS-Simons Research Enhancement Grant for Primarily Undergraduate Institution (PUI) Faculty.
S. Fallat's research was supported in part by NSERC research grant RGPIN 2025-05272. H.~Gupta acknowledges partial support from the PIMS (Pacific Institute for the Mathematical Sciences) Postdoctoral Fellowship.
J. C.-H. Lin was supported by the National Science and Technology Council of Taiwan (NSTC-113-2115-M-110-010-MY3).

\end{document}